\documentclass[10pt,twoside]{article}
\usepackage{amsmath, amssymb, amsthm}
\def\YEAR{\year}\newcount\VOL\VOL=\YEAR\advance\VOL by-1995
\def\firstpage{831}\def\lastpage{866}
\def\received{}\def\revised{}
\def\communicated{}

\makeatletter
\def\magnification{\afterassignment\m@g\count@}
\def\m@g{\mag=\count@\hsize6.5truein\vsize8.9truein\dimen\footins8truein}
\makeatother

\font\eightrm=cmr8
\font\caps=cmcsc10                    
\font\Caps=cmcsc10 scaled \magstep1   

\makeatletter
\@specialpagefalse\headheight=8.5pt

\renewcommand{\@evenhead}{%
    \ifnum\thepage>\lastpage\rlap{\thepage}\hfill%
    \else\rlap{\thepage}\slshape\leftmark\hfill{\caps\SAuthor}\hfill\fi}%
\renewcommand{\@oddhead}{%
    \ifnum\thepage=\firstpage{\hfill\llap{\thepage}}%
    \else{\slshape\rightmark}\hfill{\caps\STitle}\hfill\llap{\thepage}\fi}%
\makeatother

\def\TSkip{\bigskip}
\newbox\TheTitle{\obeylines\gdef\GetTitle #1
\ShortTitle  #2
\SubTitle    #3
\Author      #4
\ShortAuthor #5
\EndTitle
{\setbox\TheTitle=\vbox{\baselineskip=20pt\let\par=\cr\obeylines%
\halign{\centerline{\Caps##}\cr\noalign{\medskip}\cr#1\cr}}%
	\copy\TheTitle\TSkip\TSkip%
\def\next{#2}\ifx\next\empty\gdef\STitle{#1}\else\gdef\STitle{#2}\fi%
\def\next{#3}\ifx\next\empty%
    \else\setbox\TheTitle=\vbox{\baselineskip=20pt\let\par=\cr\obeylines%
    \halign{\centerline{\caps##} #3\cr}}\copy\TheTitle\TSkip\TSkip\fi%
\centerline{\caps #4}\TSkip\TSkip%
\def\next{#5}\ifx\next\empty\gdef\SAuthor{#4}\else\gdef\SAuthor{#5}\fi%
\ifx\received\empty\relax
    \else\centerline{\eightrm Received: \received}\fi%
\ifx\revised\empty\TSkip%
    \else\centerline{\eightrm Revised: \revised}\TSkip\fi%
\ifx\communicated\empty\relax
    \else\centerline{\eightrm Communicated by \communicated}\fi\TSkip\TSkip%
\catcode'015=5}}\def\Title{\obeylines\GetTitle}
\def\Abstract{\begingroup\narrower
    \parskip=\medskipamount\parindent=0pt{\caps Abstract. }}
\def\EndAbstract{\par\endgroup\TSkip}

\long\def\MSC#1\EndMSC{\def\arg{#1}\ifx\arg\empty\relax\else
     {\par\narrower\noindent%
     2010 Mathematics Subject Classification: #1\par}\fi}

\long\def\KEY#1\EndKEY{\def\arg{#1}\ifx\arg\empty\relax\else
	{\par\narrower\noindent Keywords and Phrases: #1\par}\fi\TSkip}

\long\def\THANKS#1\EndTHANKS{\def\arg{#1}\ifx\arg\empty\relax\else
	{\par\narrower\noindent#1\par}\fi\TSkip}

\newbox\TheAdd\def\Addresses{\vfill\copy\TheAdd\vfill
    \ifodd\number\lastpage\vfill\eject\phantom{.}\vfill\eject\fi}
{\obeylines\gdef\GetAddress #1
\Address #2
\Address #3
\Address #4
\EndAddress
{\def\xs{4.7truecm}\parindent=0pt
\setbox0=\vtop{{\obeylines\hsize=\xs#1\par}}\def\next{#2}
\ifx\next\empty 
     \setbox\TheAdd=\hbox to\hsize{\hfill\copy0\hfill}
\else\setbox1=\vtop{{\obeylines\hsize=\xs#2\par}}\def\next{#3}
\ifx\next\empty 
     \setbox\TheAdd=\hbox to\hsize{\hfill\copy0\hfill\copy1\hfill}
\else\setbox2=\vtop{{\obeylines\hsize=\xs#3\par}}\def\next{#4}
\ifx\next\empty\ 
     \setbox\TheAdd=\vtop{\hbox to\hsize{\hfill\copy0\hfill\copy1\hfill}
                \vskip20pt\hbox to\hsize{\hfill\copy2\hfill}}
\else\setbox3=\vtop{{\obeylines\hsize=\xs#4\par}}
     \setbox\TheAdd=\vtop{\hbox to\hsize{\hfill\copy0\hfill\copy1\hfill}
	        \vskip20pt\hbox to\hsize{\hfill\copy2\hfill\copy3\hfill}}
\fi\fi\fi\catcode'015=5}}\gdef\Address{\obeylines\GetAddress}

\newtheorem{thm}{Theorem}[section]
\newtheorem{cor}[thm]{Corollary}
\newtheorem{lem}[thm]{Lemma}
\newtheorem{prop}[thm]{Proposition}

\theoremstyle{definition}
\newtheorem{dfn}[thm]{Definition}

\newtheorem{ntn}[thm]{Notation}

\theoremstyle{remark}
\newtheorem{rmk}[thm]{Remark}

\numberwithin{equation}{section}

\newcommand{\CC}{\mathbb{C}}
\newcommand{\NN}{\mathbb{N}}

\newcommand{\TT}{\mathbb{T}}
\newcommand{\ZZ}{\mathbb{Z}}
\newcommand{\RR}{\mathbb{R}}

\newcommand{\Bb}{\mathcal{B}}
\newcommand{\Ee}{\mathcal{E}}
\newcommand{\Ff}{\mathcal{F}}
\newcommand{\Hh}{\mathcal{H}}
\newcommand{\Kk}{\mathcal{K}}
\newcommand{\Mm}{\mathcal{M}}
\newcommand{\Tt}{\mathcal{T}}
\newcommand{\Zz}{\mathcal{Z}}

\newcommand{\Ad}{\operatorname{Ad}}
\newcommand{\Aut}{\operatorname{Aut}}
\newcommand{\clsp}{\overline{\lsp}}
\newcommand{\Ext}{\operatorname{Ext}}
\newcommand{\FE}{\operatorname{FE}}
\newcommand{\Hper}{H_{\Per}}
\newcommand{\id}{\operatorname{id}}

\newcommand{\lsp}{\operatorname{span}}
\newcommand{\MCE}{\operatorname{MCE}}
\newcommand{\Obj}{\operatorname{Obj}}
\newcommand{\Per}{\operatorname{Per}}
\newcommand{\SHS}[1]{\operatorname{\mbox{$\mathrm{SH}_{#1}{\times}S$}}}
\newcommand{\tE}{\tilde{E}}
\newcommand{\tF}{\tilde{F}}

\newcommand{\Zcat}[1]{\underline{Z}^{#1}}
\newcommand{\dcat}[1]{\underline{\delta}^{#1}}

\begin{document}
\Title Twisted $C^*$-Algebras Associated
to Finitely Aligned Higher-Rank Graphs
\ShortTitle Twisted $C^*$-Algebras of Finitely Aligned $k$-Graphs
\SubTitle
\Author Aidan Sims, Benjamin Whitehead, and Michael F. Whittaker
\ShortAuthor Sims, Whitehead, and Whittaker
\EndTitle
\Abstract We introduce twisted relative Cuntz-Krieger algebras associated to finitely
aligned higher-rank graphs and give a comprehensive treatment of their fundamental
structural properties. We establish versions of the usual uniqueness theorems and the
classification of gauge-invariant ideals. We show that all twisted relative Cuntz-Krieger
algebras associated to finitely aligned higher-rank graphs are nuclear and satisfy the
UCT, and that for twists that lift to real-valued cocycles, the $K$-theory of a twisted
relative Cuntz-Krieger algebra is independent of the twist. In the final section, we
identify a sufficient condition for simplicity of twisted Cuntz-Krieger algebras
associated to higher-rank graphs which are not aperiodic. Our results indicate that this
question is significantly more complicated than in the untwisted setting.%
\EndAbstract
\MSC 46L05%
\EndMSC
\KEY $C^*$-algebra; graph algebra; Cuntz-Krieger algebra
\EndKEY
\THANKS This research was supported by the Australian Research Council.
\EndTHANKS
\Address Aidan Sims,
Benjamin Whitehead,
Michael F. Whittaker
School of Mathematics
\ and Applied Statistics
The University of Wollongong
NSW  2522
AUSTRALIA
\Address
\Address
\Address
\EndAddress 

\section{Introduction}

In \cite{KumjianPask:NYJM00}, Kumjian and Pask introduced higher-rank graphs (or
$k$-graphs) and their $C^*$-algebras. They considered only higher-rank graphs which are
\emph{row-finite} and have \emph{no sources}, the same simplifying assumptions that were
made in the first papers on graph $C^*$-algebras \cite{KumjianPaskEtAl:JFA97,
KumjianPaskEtAl:PJM98, BatesPaskEtAl:NYJM00}. The theory was later expanded
\cite{RaeburnSimsEtAl:JFA04} to include the more general finitely aligned $k$-graphs; in
this case the $C^*$-algebraic relations that describe the Cuntz-Krieger algebra were
determined by first analysing an associated Toeplitz algebra \cite{RaeburnSims:JOT05}.
Interpolating between the Toeplitz algebra and the Cuntz-Krieger algebra are the relative
Cuntz-Krieger algebras, which were introduced in \cite{Sims:IUMJ06}, and then used in
\cite{Sims:CJM06} to determine the gauge-invariant-ideal structure of the Cuntz-Krieger
algebras of finitely aligned $k$-graphs. Simplicity of the Cuntz-Krieger algebra of a
finitely aligned higher-rank graph was completely characterised in
\cite{LewinSims:MPCPS10}.

In \cite{KumjianPaskEtAl:TAMSxx}, Kumjian, Pask and Sims studied the structure theory of
the twisted Cuntz-Krieger algebra $C^*(\Lambda, c)$ associated to a row-finite
higher-rank graph $\Lambda$ with no sources and a $\TT$-valued categorical $2$-cocycle
$c$ on $\Lambda$. They subsequently proved \cite{KumjianPaskEtAl:JMAA13} that for
cocycles $c$ that lift to $\RR$-valued cocycles, the $K$-theory of $C^*(\Lambda, c)$ is
the same as that of $C^*(\Lambda)$. In this paper, we extend these results to finitely
aligned $k$-graphs, and identify the gauge-invariant ideal structure of twisted relative
Cuntz-Krieger algebras of higher-rank graphs (the Cuntz-Krieger algebra and the Toeplitz
algebra are special cases). We also establish a sufficient condition for simplicity of
$C^*(\Lambda, c)$ when $\Lambda$ is row-finite with no sources and $c$ is induced by the
degree map from a cocycle on $\ZZ^k$. The sufficient condition for simplicity is new and
requires an intricate analysis of the subalgebra generated by spanning elements in
$C^*(\Lambda, c)$ whose initial and final projections coincide, and relies on a local
decomposition of this subalgebra as a direct sum of noncommutative tori.

We have organised this paper as follows. Section~\ref{sec:background} introduces the
necessary background about $k$-graphs and their cohomology. In Section
\ref{sec:uniqueness thms}, we introduce the twisted Toeplitz algebra, the twisted
relative Cuntz-Krieger algebras and the twisted Cuntz-Krieger algebra of a finitely
aligned $k$-graph with respect to a $\TT$-valued cocycle $c$. Following the program of
\cite{RaeburnSims:JOT05, Sims:CJM06}  we prove a version of Coburn's theorem for the
twisted Toeplitz algebra, and versions of an Huef and Raeburn's gauge-invariant
uniqueness theorem and the Cuntz-Krieger uniqueness theorem for the twisted relative
Cuntz-Krieger algebras. We give a sufficient condition for the twisted Cuntz-Krieger
algebra to be simple and purely infinite. In Section~\ref{sec:ideals}, we adapt the
analysis of \cite{Sims:CJM06} to give a complete listing of the gauge-invariant ideals in
a twisted relative Cuntz-Krieger algebra. This is new even in the untwisted setting,
since the results of \cite{Sims:CJM06} only apply to the Cuntz-Krieger algebra. In
Section~\ref{sec:nuclear}, we modify arguments of \cite[\S 8]{Sims:CJM06} to show that
every twisted relative Cuntz-Krieger algebra is Morita equivalent to a crossed product of
an AF algebra by $\ZZ^k$, and is therefore nuclear and satisfies the UCT. In
Section~\ref{sec:K-theory}, we combine ideas from \cite{RaeburnSimsEtAl:JFA04} and
\cite{KumjianPaskEtAl:JMAA13} to show that if the twisting cocycle arises by
exponentiation of a real-valued $2$-cocycle, then the $K$-groups of the twisted relative
Cuntz-Krieger algebra are isomorphic to those of the corresponding untwisted algebra.

Since our proofs of the results discussed in the preceding paragraph follow the lines of
proofs of earlier results, our treatment is mostly quite brief, relying heavily on
reference to the existing arguments, and we provide additional detail only where a
nontrivial change is needed. One (perhaps surprising) example of the latter is the matter
of ascertaining which diagonal projections in a twisted relative Cuntz-Krieger algebra
are nonzero; it turns out that neither the arguments used for the row-finite case in
\cite{KumjianPaskEtAl:TAMSxx} nor those used for the untwisted case in \cite{Sims:IUMJ06}
can easily be adapted to our setting, so we use a different approach employing filters
(see \cite{Exel:SF09} and \cite{BrownloweSimsEtAl:IJM13}) in a $k$-graph $\Lambda$ and
the path-space representation. This approach simplifies and streamlines even the
untwisted setting \cite{Sims:IUMJ06}, and substantially improves upon the argument used
for twisted $C^*$-algebras of row-finite $k$-graphs with no sources in
\cite{KumjianPaskEtAl:TAMSxx}.

In the final section, we consider simplicity of twisted $k$-graph $C^*$-algebras. For
untwisted $C^*$-algebras it is proved in \cite{LewinSims:MPCPS10} that the Cuntz-Krieger
algebra of a finitely aligned $k$-graph $\Lambda$ is simple if and only if $\Lambda$ is
cofinal and aperiodic (other relative Cuntz-Krieger algebras are never simple). In the
twisted situation, the ``if" implication in this result follows from more or less the
same argument (see Corollary~\ref{cor:CKUT}); and necessity of cofinality also persists,
although the argument of \cite{LewinSims:MPCPS10} requires some modification. However,
for twisted $C^*$-algebras, aperiodicity of $\Lambda$ is not necessary for simplicity of
$C^*(\Lambda, c)$: when $\NN^2$ is regarded as a $2$-graph it is certainly not aperiodic,
but its twisted $C^*$-algebras are the rotation algebras (see
\cite[Example~7.7]{KumjianPaskEtAl:JFA12}), whose simplicity or otherwise depends on the
twisting cocycle \cite{Slawny:CMP72}. Recent work on primitive ideals in $k$-graph
$C^*$-algebras \cite{CarlsenKangEtAl:xx13} shows that if $\Lambda$ is row-finite with no
sources and is cofinal, then the primitive ideals of $C^*(\Lambda)$ are indexed by
characters of a subgroup $\Per(\Lambda)$ of $\ZZ^k$. We consider cofinal row-finite
$k$-graphs with no sources, and cocycles which are pulled back along the degree map from
cocycles $c$ of $\ZZ^k$. We show that if the associated skew-symmetric bicharacter $cc^*$
restricts to a nondegenerate bicharacter of $\Per(\Lambda)$ (see \cite{OPT,
Phillips:xx06, Slawny:CMP72}) --- the condition which characterises simplicity of the
noncommutative torus $C^*(\Per(\Lambda), c)$ --- then the associated twisted $k$-graph
$C^*$-algebra is simple.

\smallskip

We thank the anonymous referee for detailed and helpful comments; in particular the
referee's suggestions have substantially improved the presentation of the results in
Section~\ref{sec:simplicity}.

\section{Background}\label{sec:background}

Throughout this paper, $\NN$ denotes the natural numbers including $0$, and $\NN^k$ is
the monoid of $k$-tuples of natural numbers under coordinatewise addition. We denote the
generators of $\NN^k$ by $e_1, \dots, e_k$ and we write $n_i$ for the
$i$\textsuperscript{th} coordinate of $n \in \NN^k$, so that $n = \sum_i n_i \cdot e_i$.
For $m,n \in \NN^k$, we write $m \vee n$ and $m \wedge n$ for their coordinatewise
maximum and minimum. We regard $\NN^k$ as a partially ordered set with $m \le n$ if and
only if $m_i \le n_i$ for all $i$. When convenient we will regard $\NN^k$ as a category
with one object and composition given by addition.

A \emph{$k$-graph} is a countable category $\Lambda$ endowed with a functor $d : \Lambda
\to \NN^k$ that satisfies the following factorisation property: whenever $d(\lambda) =
m+n$ there are unique $\mu \in d^{-1}(m)$ and $\nu \in d^{-1}(n)$ such that $\lambda =
\mu\nu$. We write $\Lambda^n$ for $d^{-1}(n)$. Since the identity morphisms of $\Lambda$
are idempotents and $d$ is a functor, the identity morphisms belong to $\Lambda^0$. Hence
the codomain and domain maps in the category $\Lambda$ determine maps $r,s : \Lambda \to
\Lambda^0$. We then have $r(\lambda)\lambda = \lambda = \lambda s(\lambda)$ for all
$\lambda$, and the factorisation property implies that $\Lambda^0 = \{\id_o : o \in
\Obj(\Lambda)\}$. Since we are thinking of $\Lambda$ as a kind of graph, we call its
morphisms \emph{paths}; the paths $\Lambda^0$ of degree 0 are called \emph{vertices}.

Given $\lambda \in \Lambda$ and $E \subseteq \Lambda$ we write $\lambda E := \{\lambda\mu
: \mu \in E, r(\mu) = s(\lambda)\}$, and $E\lambda$ is defined similarly. In particular,
if $v \in \Lambda^0$ and $n \in \NN^k$ then $v\Lambda^n = \{\lambda \in \Lambda^n :
r(\lambda) = v\}$.

For $\mu,\nu \in \Lambda$, we define
\begin{align*}
\MCE(\mu,\nu)
	&:= \{\lambda \in \Lambda^{d(\mu) \vee d(\nu)} : \lambda = \mu\alpha = \nu\beta\text{ for some } \alpha,\beta\}\\
	&= \mu\Lambda \cap \nu\Lambda \cap \Lambda^{d(\mu) \vee d(\nu)}.
\end{align*}
We say that $\Lambda$ is \emph{finitely aligned} if $\MCE(\mu,\nu)$ is always finite
(possibly empty).

Let $v \in \Lambda^0$ and $E \subseteq v\Lambda$. We say that $E$ is \emph{exhaustive} if
for every $\lambda \in v\Lambda$ there exists $\mu \in E$ such that $\MCE(\lambda,\mu)
\not= \emptyset$. Equivalently, $E$ is exhaustive if $E\Lambda \cap \lambda\Lambda \not=
\emptyset$ for all $\lambda \in v\Lambda$. Following \cite{Sims:IUMJ06}, we write
$\FE(\Lambda)$ for the collection of all finite exhaustive sets $E$ in $\Lambda$ such
that $E \cap \Lambda^0 = \emptyset$. Each $E \in \FE(\Lambda)$ is a subset of $v\Lambda$
for some $v$. Given any $v \in \Lambda^0$ and $E \subseteq v\Lambda$, we write $r(E) :=
v$. If $E \in \FE(\Lambda)$ with $r(E) = v$, we say $E \in v\FE(\Lambda)$. If $\Ee
\subseteq \FE(\Lambda)$ and $v \in \Lambda^0$, then $v\Ee = \Ee \cap v\FE(\Lambda)$.

Following \cite{KumjianPaskEtAl:TAMSxx}, if $A$ is an abelian group, we say that $c :
\{(\mu,\nu) \in \Lambda \times \Lambda : s(\mu) = r(\nu)\} \to A$ is a \emph{normalised
cocycle} if $c(\lambda,\mu) + c(\lambda\mu, \nu) = c(\mu,\nu) + c(\lambda, \mu\nu)$ for
all composable $\lambda,\mu,\nu$, and $c(r(\lambda),\lambda) = 0 = c(\lambda,
s(\lambda))$ for all $\lambda$. (In the special case $A = \TT$, we will write the group
operation multiplicatively and the identity element as $1$, and write $\overline{c}$ for
the inverse $\overline{c}(\lambda,\mu) = \overline{c(\lambda,\mu)}$ of $c \in
\Zcat2(\Lambda, \TT)$.) We write $\Zcat2 (\Lambda, A)$ for the collection of all
normalised cocycles, which forms a group under pointwise addition in $A$. In the
cohomology introduced in \cite{KumjianPaskEtAl:TAMSxx}, the coboundary map $\dcat1$
carries a function $b : \Lambda \to A$ to the map $\dcat1b \in \Zcat2(\Lambda, A)$ given
by $(\dcat1b)(\mu,\nu) = b(\mu) + b(\nu) - b(\mu\nu)$. The elements of the range of
$\dcat1$ are called \emph{coboundaries}. Cocycles $c, c' \in \Zcat2(\Lambda, A)$ are
\emph{cohomologous} if $c' - c$ is a coboundary.

\section{The core and the uniqueness theorems}\label{sec:uniqueness thms}

In this section we prove uniqueness theorems for twisted relative $k$-graph
$C^*$-algebras: a version of Coburn's theorem for the twisted Toeplitz algebra, and
versions of the gauge-invariant uniqueness theorem and Cuntz-Krieger uniqueness theorem
for twisted relative Cuntz-Krieger algebras. We finish with a sufficient condition for
the twisted Cuntz-Krieger algebra of a $k$-graph to be simple and purely infinite.

The following definition parallels \cite[Definition~7.1]{RaeburnSims:JOT05}.

\begin{dfn}\label{dfn:TCKfamily}
Let $\Lambda$ be a finitely aligned $k$-graph and let $c \in \Zcat2(\Lambda, \TT)$. A
\emph{Toeplitz-Cuntz-Krieger $(\Lambda, c)$-family} in a $C^*$-algebra $A$ is a
collection $t = \{t_\lambda : \lambda \in \Lambda\} \subseteq A$ satisfying
\begin{itemize}
\item[(TCK1)] $\{t_v : v \in \Lambda^0\}$ is a set of mutually orthogonal
    projections;
\item[(TCK2)] $t_\mu t_\nu = c(\mu,\nu) t_{\mu\nu}$ whenever $s(\mu) = r(\nu)$;
\item[(TCK3)] $t^*_\lambda t_\lambda = t_{s(\lambda)}$ for all $\lambda \in \Lambda$;
    and
\item[(TCK4)] $t_\mu t^*_\mu t_\nu t^*_\nu = \sum_{\lambda \in \MCE(\mu,\nu)}
    t_\lambda t^*_\lambda$ for all $\mu,\nu \in \Lambda$, where empty sums are
    interpreted as zero.
\end{itemize}
We write $C^*(t)$ for $C^*(\{t_\lambda : \lambda \in \Lambda\}) \subseteq A$.
\end{dfn}

By \cite[Proposition~A.4]{Raeburn:Graphalgebras05}, relation~(TCK3) ensures that the
$t_\lambda$ are partial isometries, so that $t_\lambda t_\lambda^* t_\lambda = t_\lambda$
for all $\lambda$.

\begin{lem}\label{lem:TCK consequences}
Let $\Lambda$ be a finitely aligned $k$-graph, $c \in \Zcat2(\Lambda, \TT)$ and $t$ a
Toeplitz-Cuntz-Krieger $(\Lambda, c)$-family. The projections $\{t_\mu t^*_\mu : \mu \in
\Lambda\}$ pairwise commute, and $\{t_\mu t^*_\mu : \mu \in \Lambda^n\}$ is a collection
of mutually orthogonal projections for each $n \in \NN^k$. For $\mu,\nu,\eta,\zeta \in
\Lambda$ we have
\begin{align*}
t^*_\nu t_\eta
	&= \sum_{\nu\alpha = \eta\beta \in \MCE(\nu,\eta)} \overline{c(\nu,\alpha)}c(\eta,\beta) t_\alpha t^*_\beta\quad\text{ and } \\
t_\mu t^*_\nu t_\eta t^*_\zeta
	&= \sum_{\nu\alpha = \eta\beta \in \MCE(\nu,\eta)}
c(\mu,\alpha)\overline{c(\nu,\alpha)}c(\eta,\beta)\overline{c(\zeta,\beta)} t_{\mu\alpha}
t^*_{\zeta\beta}.
\end{align*}
In particular $C^*(t) = \clsp\{t_\mu t^*_\nu : s(\mu) = s(\nu)\}$.
\end{lem}
\begin{proof}
The first assertion follows from~(TCK4) as $\MCE(\mu,\nu) = \MCE(\nu,\mu)$. If $d(\mu) =
d(\nu)$ and $\mu \not= \nu$, then the factorisation property ensures that $\mu\Lambda
\cap \nu\Lambda = \emptyset$, and in particular $\MCE(\mu,\nu) = \emptyset$. Hence $t_\mu
t^*_\mu t_\nu t^*_\nu = 0$ by~(TCK4); so the $t_\mu t^*_\mu$ where $\mu \in \Lambda^m$
are mutually orthogonal.

For the first displayed equation, we use~(TCK4), (TCK2), and then~(TCK2) to calculate:
\begin{align*}
t^*_\nu t_\eta
	&= t^*_\nu (t_\nu t^*_\nu t_\eta t^*_\eta) t_\eta
	= \sum_{\nu\alpha=\eta\beta \in \MCE(\nu,\eta)} t^*_\nu t_{\nu\alpha}t^*_{\eta\beta} t_\eta \\
	&= \sum_{\nu\alpha=\eta\beta \in \MCE(\nu,\eta)}
        \overline{c(\nu,\alpha)}c(\eta,\beta) t^*_\nu t_\nu t_\alpha t^*_\beta t^*_\eta t_\eta \\
	&= \sum_{\nu\alpha = \eta\zeta \in \MCE(\nu,\eta)}
        \overline{c(\nu,\alpha)}c(\eta,\beta) t_\alpha t^*_\beta.
\end{align*}
Multiplying this expression on the left by $t_\mu$ and on the right by $t_\zeta$ and then
applying~(TCK2) yields the second displayed equation. It follows that $\clsp\{t_\mu
t^*_\nu : \mu,\nu \in \Lambda\}$ is closed under multiplication and hence equal to
$C^*(t)$, so the final assertion follows from the observation that if $s(\mu) \not=
s(\nu)$, then $t_\mu t^*_\nu = t_\mu t_{s(\mu)} t_{s(\nu)} t^*_\nu = 0$ by~(TCK3)
and~(TCK1).
\end{proof}

As a consequence of Lemma~\ref{lem:TCK consequences}, products of the form
$\prod_{\lambda \in E} (t_v - t_\lambda t^*_\lambda)$ for finite subsets $E$ of $\Lambda$
are well-defined, and so we can make the following definition.

\begin{dfn}\label{dfn:relCKfamily}
For $\Ee \subseteq \FE(\Lambda)$, we say that a Toeplitz Cuntz-Krieger $(\Lambda,
c)$-family $t$ is a \emph{relative Cuntz-Krieger $(\Lambda, c; \Ee)$-family} if
\begin{itemize}
\item[(CK)] for every $v \in E^0$ and $E \in v\Ee$, we have $\prod_{\lambda \in E}
    (t_{r(E)} - t_\lambda t^*_\lambda) = 0$.
\end{itemize}
A relative Cuntz-Krieger $(\Lambda, c; \FE(\Lambda))$-family is called a
\emph{Cuntz-Krieger $(\Lambda, c)$-family}.
\end{dfn}

\begin{rmk}
If $v \in \Lambda^0$ and $r(\lambda) = v$ then we have $t_v t_\lambda t^*_\lambda = c(v,
\lambda) t_{v\lambda} t^*_\lambda = t_\lambda t^*_\lambda$. Lemma~\ref{lem:TCK
consequences} implies in particular that $\{t_\lambda t^*_\lambda : \lambda \in
v\Lambda^n\}$ is a set of mutually orthogonal projections, so $\sum_{\lambda \in F}
t_\lambda t^*_\lambda$ is a projection for every finite $F \subseteq v\Lambda^n$. Since
$t_v \Big(\sum_{\lambda \in F} t_\lambda t^*_\lambda\Big) = \sum_{\lambda \in F}
t_\lambda t^*_\lambda$, we conclude that $t_v \ge \sum_{\lambda \in F} t_\lambda
t^*_\lambda$. So relations (TCK1)--(TCK4) (for $c \equiv 1$) imply relation~(4) of
\cite[Definition~7.1]{RaeburnSims:JOT05}, and in particular a Toeplitz-Cuntz-Krieger
$(\Lambda, 1)$-family as defined here is the same thing as a Toeplitz-Cuntz-Krieger
$\Lambda$-family in the sense of \cite{RaeburnSims:JOT05}. Lemma~\ref{lem:TCK
consequences} also shows that a relative Cuntz-Krieger $(\Lambda, 1; \Ee)$-family is the
same thing as a relative Cuntz-Krieger $\Lambda$-family in the sense of
\cite{Sims:IUMJ06}
\end{rmk}

The $t_\lambda$ in any Toeplitz-Cuntz-Krieger $(\Lambda, c)$-family are partial
isometries, and hence have norm 0 or 1, and the same is then true for the $t_\mu
t^*_\nu$. It is straightforward (following the strategy of, for example,
\cite[Propositions 1.20~and~1.21]{Raeburn:Graphalgebras05}) to show that there is a
$C^*$-algebra $\Tt C^*(\Lambda, c)$ generated by a Toeplitz-Cuntz-Krieger family $s^c_\Tt
:= \{s^c_\Tt(\lambda) : \lambda \in \Lambda\}$ which is universal in the sense that every
Toeplitz-Cuntz-Krieger $(\Lambda, c)$-family $t$ induces a homomorphism $\pi_t : \Tt
C^*(\Lambda, c) \to C^*(t)$ satisfying $\pi_t(s^c_\Tt(\lambda)) = t_\lambda$ for all
$\lambda$.

Given a set $\Ee \subseteq \FE(\Lambda)$, let $J_\Ee \subseteq \Tt C^*(\Lambda, c)$ be
the ideal generated by $\big\{\prod_{\lambda \in E} (s^c_\Tt(r(E)) - s^c_\Tt(\lambda)
s^c_\Tt(\lambda)^*) : E \in \Ee\big\}$. Then the quotient
\[
C^*(\Lambda, c; \Ee) := \Tt C^*(\Lambda, c) / J_\Ee
\]
is, by construction, universal for relative Cuntz-Krieger $(\Lambda, c; \Ee)$-families.
We denote each $s^c_\Tt(\lambda) + J_\Ee$ by $s^c_\Ee(\lambda)$, so that $s^c_\Ee :=
\{s^c_\Ee(\lambda) : \lambda \in \Lambda\}$ is a universal $(\Lambda, c;\Ee)$-family in
$C^*(\Lambda, c; \Ee)$. In the special case where $\Ee = \FE(\Lambda)$ we simply write
$s^c(\lambda)$ for $s^c_{\FE(\Lambda)}(\lambda)$, and we denote $C^*(\Lambda, c;
\FE(\Lambda))$ by $C^*(\Lambda, c)$. It follows almost immediately from the universal
property that $C^*(\Lambda, c; \Ee)$ depends only on the cohomology class of $c$:

\begin{prop}\label{prp:H2-dependence}
Let $\Lambda$ be a finitely aligned $k$-graph and let $\Ee$ be a subset of
$\FE(\Lambda)$. Suppose that $c_1, c_2 \in \Zcat2(\Lambda, \TT)$ are cohomologous; say
$c_1 = \dcat1b c_2$. Then there is an isomorphism $\pi : C^*(\Lambda, c_1; \Ee) \to
C^*(\Lambda, c_2; \Ee)$ such that $\pi(s^{c_1}_\Ee(\lambda)) =
b(\lambda)s^{c_2}_\Ee(\lambda)$ for each $\lambda \in \Lambda$.
\end{prop}
\begin{proof}
The proof is essentially that of \cite[Proposition~5.6]{KumjianPaskEtAl:TAMSxx}: the
formula $t_\lambda := b(\lambda)s^{c_2}_\Ee(\lambda)$ defines a relative Cuntz-Krieger
$(\Lambda, c_1; \Ee)$-family in $C^*(\Lambda, c_2;\Ee)$ and therefore induces a
homomorphism $\pi$ carrying each $s^{c_1}_\Ee(\lambda)$ to
$b(\lambda)s^{c_2}_\Ee(\lambda)$. Interchanging the roles of $c_1$ and $c_2$ and
replacing $b$ with $\overline{b}$ yields an inverse for $\pi$, and the result follows.
\end{proof}

The universal property of each $C^*(\Lambda, c; \Ee)$ ensures that there is a
homomorphism $\gamma^c : \TT^k \to \Aut(C^*(\Lambda, c; \Ee))$ such that
$\gamma^c_z(s^c_\Ee(\lambda)) = z^{d(\lambda)} s^c_\Ee(\lambda)$ for all $\lambda$. An
$\varepsilon/3$-argument shows that $\gamma^c$ is strongly continuous. Averaging over
$\gamma^c$ gives a faithful conditional expectation
\[
\Phi^c_\Ee : C^*(\Lambda, c; \Ee) \to C^*(\Lambda, c; \Ee)^{\gamma^c} := \{a \in
C^*(\Lambda, c; \Ee) : \gamma^c_z(a) = a\text{ for all }z\}.
\]
Since $\gamma^c_z(s^c_\Ee(\mu) s^c_\Ee(\nu)^*) = z^{d(\mu)- d(\nu)} s^c_\Ee(\mu)
s^c_\Ee(\nu)^*$ and since the $s^c_\Ee(\mu) s^c_\Ee(\nu)^*$ span a dense subspace of
$C^*(\Lambda, c; \Ee)$, we see that $\Phi^c_\Ee$ is characterised by
\[
\Phi^c_\Ee(s^c_\Ee(\mu) s^c_\Ee(\nu)^*) = \delta_{d(\mu), d(\nu)} s^c_\Ee(\mu)
s^c_\Ee(\nu)^*,
\]
and so $C^*(\Lambda, c; \Ee)^{\gamma^c} = \clsp\{s^c_\Ee(\mu) s^c_\Ee(\nu)^* : d(\mu) =
d(\nu)\}$.

For $E \subseteq r(\lambda) \Lambda$, we write
\[
\Ext(\lambda;E) := \bigcup_{\mu \in E} \{\alpha : \lambda\alpha \in \MCE(\lambda, \mu)\}.
\]

We say that a subset $\Ee$ of $\FE(\Lambda)$ is \emph{satiated} if it satisfies
\begin{itemize}
\item[(S1)] if $F \in \Ee$ and $\lambda \in r(F)\Lambda \setminus \{r(F)\}$, then $F
    \cup \{\lambda\} \in \Ee$;
\item[(S2)] if $F \in \Ee$ and $\lambda \in r(F)\Lambda \setminus F\Lambda$, then
    $\Ext(\lambda;F) \in \Ee$;
\item[(S3)] if $F \in \Ee$ and $\lambda, \lambda\lambda' \in F$ with $\lambda' \neq
    s(\lambda)$, then $F \setminus \{\lambda\lambda'\} \in \Ee$;
\item[(S4)] if $F \in \Ee$, $\lambda \in F$ and $G \in \Ee$ with $r(G) = s(\lambda)$,
    then $F \setminus \{\lambda\} \cup \lambda G \in \Ee$.
\end{itemize}
Lemma~5.3 of \cite{Sims:IUMJ06} shows that the sets constructed in (S1)--(S4) belong to
$\FE(\Lambda)$, and so $\FE(\Lambda)$ always contains $\Ee$ and satisfies (S1)--(S4).
Given $\Ee \subseteq \FE(\Lambda)$, we write $\overline{\Ee}$ for the intersection of all
satiated subsets of $\FE(\Lambda)$ containing $\Ee$. We call $\overline{\Ee}$ the
\emph{satiation} of $\Ee$.

Our (S1)--(S4) are different from those of \cite{Sims:IUMJ06}; but any set that can be
constructed by our (S1)--(S4) can be constructed from an application of the corresponding
operation from \cite{Sims:IUMJ06}, and conversely any set that can be constructed by any
of (S1)--(S4) from \cite{Sims:IUMJ06} can be obtained from finitely many applications of
the operations discussed above. So our definition of a satiated set agrees with that of
\cite{Sims:IUMJ06}.

\begin{ntn}
Let $\Lambda$ be a finitely aligned $k$-graph, let $c \in \Zcat2(\Lambda, \TT)$ and let
$t$ be a Toeplitz-Cuntz-Krieger $(\Lambda, c)$-family. For $\lambda \in \Lambda$, we
write $q_\lambda := t_\lambda t^*_\lambda$.

For $v \in \Lambda^0$ and a finite subset $E$ of $v\Lambda$, we define
\[
	\Delta(t)^E := \prod_{\lambda \in E} (q_{r(E)} - q_\lambda).
\]
In the universal algebras $C^*(\Lambda, c; \Ee)$, we write $p_{\Ee}^c(\lambda) :=
s_{\Ee}^c(\lambda) s^{c}_{\Ee}(\lambda)^* \in C^*(\Lambda, c; \Ee)$, and then
$\Delta(s_\Ee^c)^E = \prod_{\lambda \in E} \big(p_{\Ee}^c(r(E)) -
p_{\Ee}^c(\lambda)\big)$.
\end{ntn}

\begin{lem}\label{lem:Delta commutation}
Let $\Lambda$ be a finitely aligned $k$-graph, let $c \in \Zcat2(\Lambda, \TT)$ and let
$t$ be a Toeplitz-Cuntz-Krieger $(\Lambda, c)$-family. If $\emptyset \not= E \subseteq
v\Lambda$ is finite and $\mu \in v\Lambda$, then
\[
\Delta(t)^E t_\mu = t_\mu \Delta(t)^{\Ext(\mu;E)}.
\]
\end{lem}
\begin{proof}
Fix $\lambda \in E$ and use Lemma~\ref{lem:TCK consequences} to  calculate
\begin{align*}
\Delta(t)^E t_\mu
	&= \Delta(t)^{E \setminus\{\lambda\}} (q_v - t_\lambda t^*_\lambda) t_\mu t^*_{s(\mu)} \\
	&= \Delta(t)^{E \setminus\{\lambda\}} t_\mu -
		\sum_{\lambda\alpha = \mu\beta \in \MCE(\lambda,\mu)}
c(\lambda,\alpha)\overline{c(\lambda,\alpha)}c(\mu,\beta)\overline{c(s(\mu),\beta)}
t_{\lambda\alpha} t^*_{\beta}.
\end{align*}
For each $\alpha,\beta$ we have $t_{\lambda\alpha} = t_{\mu\beta} =
\overline{c(\mu,\beta)} t_\mu t_\beta$, and we deduce that
\[
\Delta(t)^E t_\mu
	= \Delta(t)^{E \setminus\{\lambda\}} t_\mu \Big(q_{s(\mu)} - \sum_{\beta \in
\Ext(\mu; \{\lambda\})} q_\beta\Big).
\]
The elements $\beta \in \Ext(\mu; \{\lambda\})$ all have degree $(d(\mu) \vee d(\lambda))
- d(\mu)$, and so the $q_\beta$ are mutually orthogonal. Hence $q_{s(\mu)} - \sum_{\beta
\in \Ext(\mu; \{\lambda\})} q_\beta = \prod_{\beta \in \Ext(\mu;\{\lambda\})} (q_{s(\mu)}
- q_\beta)$. Now an induction on $|E|$ using that $\Ext(\mu; E) = \bigcup_{\lambda \in E}
\Ext(\mu; \{\lambda\})$ proves the lemma.
\end{proof}

Proposition~3.12 of \cite{FarthingMuhlyEtAl:SF05} implies that $\Ext(\mu\nu; E) =
\Ext(\nu; \Ext(\mu; E))$ for all composable $\mu,\nu$. If $\mu \in E\Lambda$, then
$\Delta(s_\Tt^c)^E \leq s_\Tt^c (r(\mu))-p_\Tt^c(\mu)$, forcing $\Delta(s_\Tt^c)^E
s_\Tt^c(\mu)=0$. So Lemma~\ref{lem:Delta commutation} and condition~(S2) imply that for a
satiated subset $\Ee \subseteq \FE(\Lambda)$, the ideal $J_\Ee$ generated by
$\{\Delta(s_\Tt^c)^E : E \in \Ee\}$ satisfies
\begin{equation}\label{eq:JE descr}
J_\Ee = \clsp\{s_\Tt^c(\mu) \Delta(s_\Tt^c)^E s_\Tt^c(\nu)^* : s(\mu) = s(\nu)\text{ and
} E \in s(\mu)\Ee\}.
\end{equation}

We introduce the notion of a filter on a $k$-graph. For details, see
\cite[Section~3]{BrownloweSimsEtAl:IJM13} (taking $P = \NN^k$). We call a subset $S$ of
$\Lambda$ a \emph{filter} if
\begin{itemize}
\item[(F1)] $\lambda \Lambda \cap S \neq \emptyset$ implies $\lambda \in S$; and
\item[(F2)] $\mu,\nu \in S$ implies $\MCE(\mu,\nu) \cap S \not= \emptyset$.
\end{itemize}
For $\mu,\nu \in \Lambda$, we write $\mu \le \nu$ if $\nu = \mu\nu'$. Since $\lambda \in
\MCE(\mu,\nu)$ implies $\mu,\nu \le \lambda$, Condition~(F2) implies that each $(S, \le)$
is a directed set. Furthermore, for $\mu,\nu \in \Lambda$, distinct elements of
$\MCE(\mu,\nu)$ have the same degree, and so themselves have no common extensions. Hence
condition~(F2) implies that $\!|\MCE(\mu,\nu) \cap S| = 1$ for $\mu,\nu \in S$.

For a satiated set $\Ee \subseteq \FE(\Lambda)$, we say that a filter $S$ is
\emph{$\overline{\Ee}$}-compatible if, whenever $\lambda \in S$ and $E \in
s(\lambda)\overline{\Ee}$, we have $\lambda E \cap S \neq \emptyset$.

\begin{rmk}\label{rmk:filters vs paths}
It is straightforward to check that the $\overline{\Ee}$-compatible boundary paths of
\cite{Sims:IUMJ06} are in bijection with $\overline{\Ee}$-compatible filters via the map
$x \mapsto \{x(0,n) : n \le d(x)\}$.
\end{rmk}

Let $\{h_\lambda : \lambda \in \Lambda\}$ be the usual orthonormal basis for
$\ell^2(\Lambda)$. Routine calculations show that the formula $T_\mu h_\nu :=
\delta_{s(\mu), r(\nu)} c(\mu,\nu) h_{\mu\nu}$ determines a Toeplitz-Cuntz-Krieger
$(\Lambda, c)$-family in $\Bb(\ell^2(\Lambda))$. Let $\pi_T : \Tt C^*(\Lambda, c) \to
\Bb(\ell^2(\Lambda))$ be the representation induced by the universal property of $\Tt
C^*(\Lambda, c)$.

\begin{prop}\label{prp:nonzero elements}
Let $\Lambda$ be a finitely aligned $k$-graph and let $c \in \Zcat2(\Lambda, \TT)$. Let
$\Ee \subseteq \FE(\Lambda)$. Then the $s_\Ee^c(v)$ are all nonzero, and for each $v \in
\Lambda^0$ and finite $F \subseteq v\Lambda$, we have $\Delta(s_\Ee^c)^F = 0$ if and only
if either $v \in F$ or $F \in \overline{\Ee}$. We have $J_\Ee = J_{\overline{\Ee}}$ and
$C^*(\Lambda, c; \Ee) = C^*(\Lambda, c; \overline{\Ee})$.
\end{prop}

To prove the proposition, we need to be able to tell when  $a \in \Tt C^*(\Lambda, c)$
does not belong to $J_\Ee$. We present the requisite statement as a separate result
because we will use it again in Section~\ref{sec:simplicity} (see
Lemma~\ref{lem:cofinality necessary}).

\begin{prop}\label{prp:->0 on filters}
Let $\Lambda$ be a finitely aligned $k$-graph and let $c \in \Zcat2(\Lambda, \TT)$. Let
$\Ee \subseteq \FE(\Lambda)$. Then
\[\textstyle
J_\Ee = \{a \in \Tt C^*(\Lambda, c) :
    \lim_{\lambda \in S} \pi_T(a)h_\lambda = 0\text{ for every $\overline{\Ee}$-compatible
    filter $S$}\}.
\]
\end{prop}

Here we only need the ``$\subseteq$" containment --- we will use it immediately to
establish Proposition~\ref{prp:nonzero elements}, which we will use in turn to prove the
gauge-invariant uniqueness theorem and then our characterisation of the gauge-invariant
ideals in $C^*(\Lambda, c; \Ee)$. With this catalogue of gauge-invariant ideals in hand,
it is then easy to establish the reverse inclusion. So we prove the ``$\subseteq$"
inclusion now, and defer the proof of the ``$\supseteq$" inclusion until after
Theorem~\ref{thm:g-i ideals}.

\begin{proof}[Proof of $\subseteq$ in Proposition~\ref{prp:->0 on filters}]
Fix $v \in \Lambda^0$, an element $E \in \Ee$ with $r(E) = v$ and an
$\overline{\Ee}$-compatible filter $S$. We claim that $\Delta(T)^E h_\lambda = 0$ for
large $\lambda \in S$. To see this, we consider two cases. First suppose that $v \not\in
S$. Then~(F1) ensures that $S \cap v\Lambda =\emptyset$ and so $T_v h_\lambda = 0$ for
all $\lambda \in S$. Since $\Delta(T)^E \le T_v$, it follows that $\Delta(T)^E h_\lambda
= 0$ for all $\lambda \in S$. Now suppose that $v \in S$. Since $S$ is
$\overline{\Ee}$-compatible and $E \in \Ee$, there exists $\lambda \in E \cap S$. Hence
$\Delta(T)^E \le T_v - T_\lambda T^*_\lambda$. For $\mu \in S$ with $\lambda \le \mu$, we
therefore have $\Delta(T)^E h_\mu = 0$. Hence $\Delta(T)^E h_\lambda = 0$ for large
$\lambda \in S$, proving the claim. It is now elementary to deduce that any finite linear
combination of the form $a = \sum_{\mu,\nu \in F} a_{\mu,\nu} s^c_\Tt(\mu)
\Delta(s^c_\Tt)^E s^c_\Tt(\nu)^*$ where the $E$ belong to $\Ee$ satisfies $\pi_T(a)
h_\lambda = 0$ for large $\lambda \in E$. A continuity argument (details appear in
\cite[Lemma~3.4.14]{BenThesis}) using this and~\eqref{eq:JE descr} then completes the
proof.
\end{proof}

\begin{lem}\label{lem:nonzero under piT}
Let $\Lambda$ be a finitely aligned $k$-graph and let $c \in \Zcat2(\Lambda, \TT)$. Let
$\Ee \subseteq \FE(\Lambda)$. For each $v \in \Lambda^0$ there exists an
$\overline{\Ee}$-compatible filter $S$ such that $\|T_v h_\lambda\| = 1$ for all $\lambda
\in S$, and for each $E \in \FE(\Lambda) \setminus \overline{\Ee}$, there exists an
$\overline{\Ee}$-compatible filter $S$ such that $\|\Delta(T)^E h_\lambda\| = 1$ for all
$\lambda \in S$.
\end{lem}
\begin{proof}
Fix $v \in \Lambda^0$. The argument of \cite[Lemma~4.7]{Sims:IUMJ06} shows that there
exists an $\overline{\Ee}$-compatible filter $S$ that contains $v$. Hence $\|T_v
h_\lambda\| = \|h_\lambda\| = 1$ for all $\lambda \in S$. Now fix $E \in \FE(\Lambda)
\setminus \Ee$. Then the argument of \cite[Lemma~4.7]{Sims:IUMJ06} implies that there is
an $\overline{\Ee}$-compatible filter $S$ such that $r(E) \in S$ but $E\Lambda \cap S =
\emptyset$. Thus $T_\mu T^*_\mu h_\lambda = 0$ for all $\mu \in E$ and $\lambda \in S$,
and so $\|\Delta(T)^E h_\lambda\| = \|T_{r(E)} h_\lambda\| = 1$ for all $\lambda \in S$.
\end{proof}

\begin{proof}[Proof of Proposition~\ref{prp:nonzero elements}]
For the ``if" implication, observe that if $v \in F$ then $\Delta_v \le q_v - q_v = 0$.
To see that $F \in \overline{\Ee}$ implies $\Delta(s_\Ee^c)^F = 0$, one checks that the
calculations in \cite{Sims:IUMJ06} that establish the corresponding statement for the
untwisted $C^*$-algebra $C^*(\Lambda; \Ee)$ are also valid in the twisted $C^*$-algebra.
Since we clearly have $J_{\Ee} \subseteq J_{\overline{\Ee}}$ we deduce from this that
$J_\Ee = J_{\overline{\Ee}}$, and therefore that $C^*(\Lambda, c; \Ee) = C^*(\Lambda, c;
\overline{\Ee})$.

For the ``only if'' implication, first observe that we have just seen that
$J_{\overline{\Ee}} = J_{\Ee}$. So it suffices to show that $s_\Tt^c(v) \not \in
J_{\overline{\Ee}}$ for all $v \in \Lambda^0$ and that if $E \subset v\Lambda$ is finite,
does not contain $v$ and does not belong to $\overline{\Ee}$, then $\Delta(s_\Tt^c)^E
\not\in J_{\overline{\Ee}}$.

The ``$\subseteq$" containment in Proposition~\ref{prp:->0 on filters} and the first
assertion of Lemma~\ref{lem:nonzero under piT} combine to show that $s_\Tt^c(v) \not\in
J_{\overline{\Ee}}$ for each $v \in \Lambda^0$. Likewise, the ``$\subseteq$" containment
in Proposition~\ref{prp:->0 on filters} combined with the second assertion of
Lemma~\ref{lem:nonzero under piT} shows that $\Delta(s_\Tt^c)^E \not \in
J_{\overline{\Ee}}$ for all $E \in \FE(\Lambda) \setminus \overline{\Ee}$. Finally,
suppose that $F \subseteq v\Lambda \setminus \{v\}$ is finite and does not belong to
$\FE(\Lambda)$. Then there exists $\lambda \in v\Lambda$ such that $\Ext(\lambda; F) =
\emptyset$. It is not hard to check (see, for example, the argument of
\cite[Lemma~4.4]{Sims:IUMJ06}) that $\{\mu : \mu\mu' \in \lambda S \text{ for some }
\mu'\}$ is an $\overline{\Ee}$-compatible filter which does not intersect $F$. So as
above, we have $\|\pi_T(\Delta(s_\Tt^c)^F) h_\lambda\| = 1$ for all $\lambda \in S$ and
yet another application of the ``$\subseteq$" containment in Proposition~\ref{prp:->0 on
filters} implies that $\Delta(s_\Tt^c)^F \not \in J_{\overline{\Ee}}$.
\end{proof}

Given a finite subset $E$ of $\Lambda$ and an element $\mu \in E$, we write $T(E;\mu) :=
\{\mu' \in s(\mu)\Lambda\setminus \{s(\mu)\} : \mu\mu' \in E\}$.

Recall from \cite{RaeburnSimsEtAl:JFA04} that if $E$ is a finite subset of a finitely
aligned $k$-graph $\Lambda$, then there is a finite subset $F$ of $\Lambda$ which
contains $E$ and has the property that if $\mu,\nu,\sigma, \tau \in F$ with $d(\mu) =
d(\nu)$ and $d(\sigma) = d(\tau)$, then $\mu\alpha$ and $\tau\beta$ belong to $F$ for
every $\nu\alpha = \sigma\beta \in \MCE(\nu,\sigma)$. The smallest such set is denoted
$\Pi E$ and is closed under minimal common extensions (just take $\mu = \nu$ and $\tau =
\sigma$). Suppose that $E = \Pi E$. The defining property of $\Pi E$ ensures that $T(E;
\mu) = T(E; \nu)$ whenever $\mu,\nu \in E$ satisfy $d(\mu) = d(\nu)$. Fix $c \in
\Zcat2(\Lambda, \TT)$ and a Toeplitz-Cuntz-Krieger $(\Lambda, c)$-family $t$. For $\mu,
\nu \in E$ (possibly equal) with $d(\mu) = d(\nu)$, we write
\[
\Theta(t)^E_{\mu,\nu} = t_\mu \Delta(t)^{T(E;\mu)} t_\nu^*.
\]

\begin{lem}\label{lem:mx units}
Let $\Lambda$ be a finitely aligned $k$-graph, let $c \in \Zcat2(\Lambda, \TT)$ and let
$t$ be a Toeplitz-Cuntz-Krieger $(\Lambda, c)$-family. Suppose that $E \subseteq \Lambda$
is finite and satisfies $E = \Pi E$. Then $M(t)_E := \lsp\{t_\mu t^*_\nu : \mu,\nu \in E,
d(\mu)= d(\nu)\}$ is a finite dimensional $C^*$-subalgebra of $C^*(t)$, and
$\{\Theta(t)^E_{\mu,\nu} : \mu,\nu \in E, d(\mu) = d(\nu), s(\mu)=s(\nu),
\Delta(t)^{T(E;\mu)} \not= 0\}$ is a family of nonzero matrix units spanning $M(t)_E$.
\end{lem}
\begin{proof}
Using Lemma~\ref{lem:TCK consequences}, it is easy to see that $\lsp\{t_\mu t^*_\nu :
\mu,\nu \in E, d(\mu) = d(\nu)\}$ is closed under multiplication and hence a subalgebra
of $C^*(t)$. Since $T(E;\mu) = T(E;\nu)$ whenever $d(\mu) = d(\nu)$ and $s(\mu)=s(\nu)$,
we have $(\Theta(t)^E_{\mu,\nu})^* = \Theta(t)^E_{\nu,\mu}$, so $M(t)_E$ is a
$^*$-subalgebra of $C^*(t)$. It is finite-dimensional by definition, and then it is
automatically norm-closed and hence a $C^*$-subalgebra.

The argument of \cite[Lemma~3.9(ii)]{RaeburnSimsEtAl:JFA04}, using Lemma~\ref{lem:Delta
commutation} twice in place of \cite[Lemma~3.10]{RaeburnSimsEtAl:JFA04} shows that
$\Theta(t)^E_{\mu,\nu} \Theta(t)^E_{\eta,\xi} = \delta_{\nu,\eta} \Theta(t)^E_{\mu,\xi}$.
Since $\Delta(t)^{T(E;\mu)} = t^*_\mu \Theta(t)^E_{\mu,\nu} t_\nu$ and since
$\Theta(t)^E_{\mu,\nu}  = t_\mu \Delta(t)^{T(E;\mu)} t^*_\nu$, we have
$\Theta(t)^E_{\mu,\nu} = 0 \iff \Delta(t)^{T(E;\mu)} = 0$. The arguments of
\cite[Proposition~3.5 and Corollary~3.7]{RaeburnSimsEtAl:JFA04} only involve products of
elements of the form $t_\mu t^*_\mu$, and these satisfy the same relations in a
Toeplitz-Cuntz-Krieger $(\Lambda, c)$-family as in a Toeplitz-Cuntz-Krieger
$\Lambda$-family. So the proof of \cite[Corollary~3.7]{RaeburnSimsEtAl:JFA04} implies
that $t_\mu t^*_\mu = \sum_{\mu\mu' \in E} \Delta(t)^{T(E;\mu\mu')}$ for $\mu \in E$. Now
we follow the proof of \cite[Corollary~3.11]{RaeburnSimsEtAl:JFA04} to see that
\begin{equation}\label{eq:tmutnu=thetasum}
t_\mu t^*_\nu = \sum_{\mu\alpha \in E} c(\mu, \alpha)\overline{c(\nu,\alpha)}
\Theta(t)^E_{\mu\alpha,\nu\alpha}.
\end{equation}
Hence the $\Theta(t)^E_{\mu,\nu}$ span $M(t)_E$.
\end{proof}

\begin{thm}\label{thm:core injectivity}
Let $\Lambda$ be a finitely aligned $k$-graph, let $c \in \Zcat2(\Lambda, \TT)$, and let
$\Ee$ be a satiated subset of $\FE(\Lambda)$. Suppose that $t$ is a relative
Cuntz-Krieger $(\Lambda, c; \Ee)$-family. The induced homomorphism $\pi^\Ee_t :
C^*(\Lambda, c; \Ee) \to C^*(t)$ restricts to an injective homomorphism of $C^*(\Lambda,
c; \Ee)^{\gamma^c}$ if and only if every $t_v$ is nonzero, and $\Delta(t)^F \not= 0$ for
all $F \in \FE(\Lambda) \setminus \Ee$.
\end{thm}
\begin{proof}
Proposition~\ref{prp:nonzero elements} implies that if $t_v = 0$ for some $v \in E^0$ or
$\Delta(t)^F = 0$ for some $F \in \FE(\Lambda) \setminus \Ee$, then $\pi^\Ee_t$ is not
injective on $C^*(\Lambda)^\gamma$.

Now suppose that each $t_v \not= 0$ and that $\Delta(t)^F \not= 0$ whenever $F \in
\FE(\Lambda) \setminus \Ee$. Since every finite subset of $\Lambda$ is contained in a
finite subset $E$ satisfying $E = \Pi E$, we have
\[
C^*(\Lambda, c; \Ee)^{\gamma^c}
	= \overline{\bigcup_{\Pi E = E} M(s_\Ee^c)^E},
\]
and so it suffices to show that $\pi^\Ee_t$ is injective (and hence isometric) on each
$M(s_\Ee^c)^E$. Fix $E$ such that $\Pi E = E$. Since matrix algebras are simple,
Lemma~\ref{lem:mx units} implies that it suffices to show that $\Delta(t)^{T(E;\lambda)}
= 0$ implies $\Delta(s_\Ee^c)^{T(E;\lambda)} = 0$ for each $\lambda \in E$. We consider
two cases. If $T(E;\lambda) \in \FE(\Lambda)$, then $\Delta(t)^{T(E;\lambda)} = 0$
implies $T(E;\lambda) \in \Ee$ by hypothesis, and then $\Delta(s_\Ee^c)^{T(E;\lambda)} =
0$ as well. Now suppose that $T(E;\lambda) \not\in \FE(\Lambda)$. It suffices to show
that $\Delta(t)^{T(E;\lambda)} \not= 0$. Since $T(E;\lambda) \not\in \FE(\Lambda)$ and
$T(E;\lambda) \cap \Lambda^0 = \emptyset$, the set $T(E;\lambda)$ is not exhaustive. Fix
$\mu \in s(\lambda)\Lambda$ such that $\MCE(\mu,\alpha) = \emptyset$ for every $\alpha
\in T(E;\lambda)$. Since $t^*_\mu t_\mu = t_{s(\mu)}$ is nonzero, $q_\mu=t_\mu t^*_\mu$
is also nonzero. Lemma~\ref{lem:TCK consequences} implies that $q_\alpha q_\mu = 0$ for
all $\alpha \in T(E;\lambda)$. Since the $q_\eta$ all commute and $q_\mu$ is a
projection, we deduce that
\[
\Delta(t)^{T(E;\lambda)} q_\mu
	= \prod_{\alpha \in T(E;\lambda)} \big((t_{s(\lambda)} -q_\alpha)q_\mu\big)
	=q_\mu \not= 0.
\]
Hence $\Delta(t)^{T(E;\lambda)} \not= 0$.
\end{proof}

From this flow a number of uniqueness theorems, all based on the following standard idea
from~\cite{Cuntz:CMP77}.

\begin{lem}\label{lem:usual argument}
Let $A$ be a $C^*$-algebra, $\Phi : A \to A$ a faithful conditional expectation, and $\pi
: A \to B$ a $C^*$-homomorphism. Suppose that there is a linear map $\Psi : B \to B$ such
that $\Psi \circ \pi = \pi \circ \Phi$. Then $\pi$ is injective if and only if
$\pi|_{\Phi(A)}$ is injective.
\end{lem}
\begin{proof}
The ``only if" implication is obvious. For the ``if'' implication, suppose that $\pi(a) =
0$. Then $\pi(\Phi(a^*a)) = \Psi(\pi(a^*a)) = 0$. Since $\pi$ is injective on the range
of $\Phi$ and $\Phi$ is faithful, we deduce that $a = 0$.
\end{proof}

The next result is a version of Coburn's theorem for $\Tt C^*(\Lambda,c)$.

\begin{thm}
Let $\Lambda$ be a finitely aligned $k$-graph and let $c \in \Zcat2(\Lambda, \TT)$. Let
$t$ be a Toeplitz-Cuntz-Krieger $(\Lambda, c)$-family. Then the induced homomorphism
$\pi_t$ of $\Tt C^*(\Lambda, c)$ is injective if and only if every $t_v \not= 0$ and
$\Delta(t)^E \not= 0$ for every $E \in \FE(\Lambda)$.
\end{thm}
\begin{proof}
Averaging over the gauge action on $\Tt C^*(\Lambda, c)$ defines a faithful conditional
expectation $\Phi^{\gamma^c} : \Tt C^*(\Lambda, c) \to \Tt C^*(\Lambda, c)^{\gamma^c}$
that is characterised by $\Phi^{\gamma^c}(s_\Tt^c(\mu) s_\Tt^c(\nu)^*) = \delta_{d(\mu),
d(\nu)} s_\Tt^c(\mu) s_\Tt^c(\nu)^*$ for $\mu,\nu \in \Lambda$. Lemma~\ref{lem:mx units}
shows that $\Tt C^*(\Lambda, c)^{\gamma^c}$ is AF, and $\clsp\{s_\Tt^c(\mu)
s_\Tt^c(\mu)^* : \mu \in \Lambda\}$ is its canonical diagonal subalgebra. So there is a
faithful conditional expectation $\Phi^c_D : \Tt C^*(\Lambda, c)^{\gamma^c} \to
\clsp\{s_\Tt^c(\mu) s_\Tt^c(\mu)^* : \mu \in \Lambda\}$ satisfying $\Phi^c_D(s_\Tt^c(\mu)
s_\Tt^c(\nu)^*) = \delta_{\mu,\nu} s_\Tt^c(\mu) s_\Tt^c(\mu)^*$. So $\Phi := \Phi^c_D
\circ \Phi^{\gamma^c}$ is a faithful conditional expectation satisfying
$\Phi(s_\Tt^c(\mu) s_\Tt^c(\nu)^*) = \delta_{\mu,\nu} s_\Tt^c(\mu) s_\Tt^c(\mu)^*$ for
all $\mu,\nu \in \Lambda$.

The argument of \cite[Proposition~8.9]{RaeburnSims:JOT05} (this uses Lemmas 8.5--8.8 of
the same paper; all the arguments go through for twisted TCK families) gives a
norm-decreasing linear map $\Psi$ from $C^*(t)$ to $\clsp\{t_\lambda t^*_\lambda :
\lambda \in \Lambda\}$ such that $\Psi \circ \pi_t = \pi_t \circ \Phi$. Now
Lemma~\ref{lem:usual argument} proves the result.
\end{proof}

We also obtain the following version of an Huef and Raeburn's gauge-invariant uniqueness
theorem \cite{anHuefRaeburn:ETDS97}.

\begin{thm}\label{thm:giut}
Let $\Lambda$ be a finitely aligned $k$-graph, let $c \in \Zcat2(\Lambda, \TT)$, and let
$\Ee$ be a satiated subset of $\FE(\Lambda)$. Suppose that $t$ is a relative
Cuntz-Krieger $(\Lambda, c; \Ee)$-family in a $C^*$-algebra $B$ and that there is a
strongly continuous action $\beta$ of $\TT^k$ on $B$ such that $\beta_z(t_\lambda) =
z^{d(\lambda)} t_\lambda$ for all $\lambda \in \Lambda$. Then the induced homomorphism
$\pi^\Ee_t : C^*(\Lambda, c; \Ee) \to C^*(t)$ is injective if and only if every $t_v$ is
nonzero, and $\Delta(t)^F \not= 0$ for all $F \in \FE(\Lambda) \setminus \Ee$.
\end{thm}
\begin{proof}
Lemma~\ref{lem:usual argument} applied to the expectation $\Phi$ obtained from averaging
over $\gamma^c$ and the expectation $\Psi$ obtained by averaging over $\beta$ shows that
$\pi_t^\Ee$ is injective if and only if it restricts to an injection on
$C^*(\Lambda,c;\Ee)^{\gamma^c}$. So the result follows from Theorem~\ref{thm:core
injectivity}.
\end{proof}

We also obtain a version of the Cuntz-Krieger uniqueness theorem (cf.
\cite[Theorem~6.3]{Sims:IUMJ06}). Given a filter $S$ and a path $\mu$ with $s(\mu) \in
S$, we write $\ell_\mu(S)$ for the filter $\{\nu \in \Lambda : \mu S \cap \nu\Lambda
\not= \emptyset\}$. We say that a filter $S \subseteq \Lambda$ is \emph{separating} if
whenever $s(\mu) = s(\nu) \in S$ and there is a filter $T$ such that $\ell_\mu(S) \cup
\ell_\nu(S) \subseteq T$, we have $\mu = \nu$.

\begin{thm}\label{thm:CKUT}
Let $\Lambda$ be a finitely aligned $k$-graph, let $c \in \Zcat2(\Lambda, \TT)$, and let
$\Ee$ be a satiated subset of $\FE(\Lambda)$. Suppose that for every $v \in \Lambda^0$
there is a separating $\Ee$-compatible filter $S$ such that $r(S) = v$, and that for
every $F \in \FE(\Lambda) \setminus \Ee$ there is a separating $\Ee$-compatible filter
$S$ such that $r(S) = r(F)$ and $S \cap F\Lambda = \emptyset$. Suppose that $t$ is a
relative Cuntz-Krieger $(\Lambda, c; \Ee)$-family in a $C^*$-algebra $B$. Then the
induced homomorphism $\pi^\Ee_t : C^*(\Lambda, c; \Ee) \to C^*(t)$ is injective if and
only if every $t_v$ is nonzero, and $\Delta(t)^F \not= 0$ for all $F \in \FE(\Lambda)
\setminus \Ee$.
\end{thm}
\begin{proof}
Using Remark~\ref{rmk:filters vs paths}, we see that our hypothesis about the existence
of separating $\Ee$-compatible filters is equivalent to Condition~(C) of
\cite{Sims:IUMJ06}. Now one follows the proof of \cite[Theorem~6.3]{Sims:IUMJ06}; the
only place where the cocycle $c$ comes up is in the displayed calculation at the top of
\cite[page~866]{Sims:IUMJ06}, where the formula for $P_2 \Theta(t)^{\Pi E}_{\lambda, \mu}
P_2$ picks up a factor of $c(\lambda, x(0,N))\overline{c(\mu, x(0,N))}$; but the
resulting elements still form a family of matrix units, so the rest of the argument
proceeds without change.
\end{proof}

The hypothesis of Theorem \ref{thm:CKUT} simplifies significantly in the key case where
$\Ee = \FE(\Lambda)$. Recall from \cite[Definition~3.3]{LewinSims:MPCPS10} that a
finitely aligned $k$-graph $\Lambda$ is \emph{cofinal} if, for all $v,w \in \Lambda^0$
there exists a finite exhaustive subset $E$ of $v\Lambda$ (here $E$ may contain $v$) such
that $w\Lambda s(\alpha) \not= \emptyset$ for every $\alpha \in E$. Recall from
\cite{LewinSims:MPCPS10} that a $k$-graph $\Lambda$ is \emph{aperiodic} if whenever
$\mu,\nu$ are distinct paths with the same source, there exists $\tau \in s(\mu)\Lambda$
such that $\MCE(\mu\tau, \nu\tau) = \emptyset$.

\begin{cor}\label{cor:CKUT}
Let $\Lambda$ be a finitely aligned $k$-graph and let $c \in \Zcat2(\Lambda, \TT)$.
Suppose $\Lambda$ is aperiodic. Then a homomorphism $\pi : C^*(\Lambda, c) \to B$ is
injective if and only if every $\pi(s^c_v) \not = 0$. If $\Lambda$ is cofinal then
$C^*(\Lambda, c)$ is simple.
\end{cor}
\begin{proof}
Condition~(A) of \cite{FarthingMuhlyEtAl:SF05}, reinterpreted using
Remark~\ref{rmk:filters vs paths}, requires that for every $v \in \Lambda^0$ there is a
separating $\FE(\Lambda)$-compatible filter $S$ containing $v$. So the implication
(i)$\implies$(iii) of \cite[Proposition~3.6]{LewinSims:MPCPS10} implies that the
hypotheses of Theorem~\ref{thm:CKUT} are satisfied with $\Ee = \FE(\Lambda)$. This proves
the first assertion.

Now suppose that $\Lambda$ is cofinal. Then the argument of (ii)$\implies$(iii) of
\cite[Theorem~5.1]{LewinSims:MPCPS10} carries over unchanged to the twisted setting to
show that if $I$ is an ideal of $C^*(\Lambda, \Ee)$ and $s^c_v \in I$ for some $v$, then
$s^c_w \in I$ for every $w \in \Lambda^0$ and hence $I = C^*(\Lambda, c)$. So
$C^*(\Lambda, c)$ is simple by the preceding paragraph.
\end{proof}

Recall from \cite{EvansSims:JFA2012} that if $\Lambda$ is a finitely aligned $k$-graph,
then a \emph{generalised cycle} in $\Lambda$ is a pair $(\mu,\nu) \in \Lambda$ such that
$r(\mu) = r(\nu)$, $s(\mu) = s(\nu)$ and $\MCE(\mu\tau, \nu) \not= \emptyset$ for all
$\tau \in s(\mu)\Lambda$. We say that a generalised cycle $(\mu,\nu)$ has an
\emph{entrance} if there exists $\tau \in s(\nu) \Lambda$ such that $\MCE(\mu,\nu\tau) =
\emptyset$. If $c \in \Zcat2(\Lambda, \TT)$ and $t$ is a Toeplitz-Cuntz-Krieger
$(\Lambda, c)$-family such that $t_{s(\tau)} \not= 0$, then $V := t_\mu t^*_\nu$
satisfies $q_\mu = VV^* \sim V^*V = q_\nu > q_\nu-q_{\nu\tau} \geq q_\mu$, where $\sim$
denotes Murray-von Neumann equivalence. Hence $q_\mu$ is infinite.

\begin{prop}
Let $\Lambda$ be a finitely aligned $k$-graph with no sources, and let $c \in
\Zcat2(\Lambda, \TT)$. Suppose that $\Lambda$ is aperiodic and that for every $v \in
\Lambda^0$ there is a generalised cycle $(\mu,\nu)$ with an entrance such that $v \Lambda
r(\mu) \not= \emptyset$. Then every hereditary subalgebra of $C^*(\Lambda, c)$ contains
an infinite projection. If $\Lambda$ is cofinal, then $C^*(\Lambda, c)$ is simple and
purely infinite.
\end{prop}
\begin{proof}
Let $v \in \Lambda^0$ and choose a generalised cycle $(\mu,\nu)$ with an entrance such
that $v\Lambda r(\mu) \not= \emptyset$; say $\lambda \in v\Lambda r(\mu)$. Then $s^c(v)
\ge s^c(\lambda) s^c(\lambda)^* \sim s^c(\lambda)^* s^c(\lambda) \ge p^c(\mu)$ is
infinite. Using this in place of \cite[Lemma~8.13]{Sims:CJM06}, we can now follow the
proof of \cite[Proposition~8.8]{Sims:CJM06} (including the proof of
\cite[Lemma~8.12]{Sims:CJM06}).
\end{proof}

\section{Gauge-invariant ideals}\label{sec:ideals}

In this section we list the gauge-invariant ideals of each $C^*(\Lambda, c; \Ee)$. There
is by now a fairly standard program for this (the basic idea goes back to
\cite{Cuntz:IM81} and \cite{anHuefRaeburn:ETDS97}). The standard arguments are applied to
the untwisted Cuntz-Krieger algebras of finitely aligned $k$-graphs in \cite{Sims:CJM06},
and we follow the broad strokes of that treatment. The twist does not affect the
arguments much, but we need to make a few adjustments to pass from $C^*(\Lambda, c)$ to
$C^*(\Lambda, c; \Ee)$. So we give more detail here than in the preceding section.

\begin{dfn}
Let $\Lambda$ be a finitely aligned $k$-graph and let $\Ee$ be a satiated subset of
$\FE(\Lambda)$. We say that $H \subseteq \Lambda^0$ is \emph{hereditary} if $s(H\Lambda)
\subseteq H$, and that it is \emph{$\Ee$-saturated} if whenever $E \in \Ee$ and $s(E)
\subseteq H$ we have $r(E) \in H$.
\end{dfn}

Recall from \cite[Lemma~4.1]{Sims:CJM06} that if $\Lambda$ is a finitely aligned
$k$-graph and $H \subseteq \Lambda^0$ is hereditary, then $\Lambda \setminus \Lambda H$
is a finitely aligned $k$-graph under the operations and degree map inherited from
$\Lambda$.

\begin{lem}
Let $\Lambda$ be a finitely aligned $k$-graph and let $\Ee$ be a satiated subset of
$\FE(\Lambda)$. Suppose that $H \subseteq \Lambda^0$ is hereditary and $\Ee$-saturated.
Then
\[
	\Ee_H := \{E \setminus EH : E \in \Ee, r(E) \not\in H\}	
\]
is a subset of $\FE(\Lambda \setminus \Lambda H)$.
\end{lem}
\begin{proof}
No element $F$ of $\Ee_H$ contains $r(F)$ because no element $E$ of $\Ee$ contains
$r(E)$. The elements of $\Ee_H$ are nonempty because $H$ is $\Ee$-saturated. Fix $F \in
\Ee_H$, say $r(F)=v$, and fix $\lambda \in v\Lambda \setminus \Lambda H$. Choose $E \in
\Ee$ such that $F = E \setminus EH$. We have to show that there exists $\mu \in F$ such
that $\MCE(\mu,\lambda) \cap (\Lambda \setminus \Lambda H) \neq \emptyset$. We consider
two cases. First suppose that $\lambda \in E\Lambda$, say $\lambda = \mu\mu'$ with $\mu
\in E$. Since $H$ is hereditary and $s(\lambda) = s(\mu') \not\in H$, we have $s(\mu) =
r(\mu') \not\in H$, so $\mu \in F$ satisfies $\lambda \in \MCE(\mu,\lambda) \cap (\Lambda
\setminus \Lambda H)$. Now suppose that $\lambda \not\in E\Lambda$. Then $\Ext(\lambda;
E) \in \Ee$ by~(S2). Since $H$ is $\Ee$-saturated and $s(\lambda) \not \in H$, it follows
that $s(\Ext(\lambda; E)) \not\subseteq H$. So there exists $\mu \in E$ and $\mu\alpha =
\lambda\beta \in \MCE(\mu,\lambda)$ with $s(\alpha) \not\in H$. Since $H$ is hereditary,
it follows that $s(\mu) \not\in H$, so $\mu \in F$ and $\MCE(\mu,\lambda)\cap (\Lambda
\setminus \Lambda H) \neq \emptyset$.
\end{proof}

\begin{lem}\label{lem:ideal sets}
Let $\Lambda$ be a finitely aligned $k$-graph, let $c \in \Zcat2(\Lambda, \TT)$, and let
$\Ee$ be a satiated subset of $\FE(\Lambda)$. Let $I$ be an ideal of $C^*(\Lambda, c;
\Ee)$. Then
\begin{enumerate}
\item\label{it:HI}
	$H_I := \{v \in \Lambda^0 : s_\Ee^c(v) \in I\}$ is hereditary and $\Ee$-saturated;
\item\label{it:BI}
	$\Bb_I := \{F \in \FE(\Lambda \setminus \Lambda H_I) : \Delta(s_\Ee^c)^F \in I\}$ is
a satiated subset of $\FE(\Lambda \setminus \Lambda H)$; and
\item\label{it:E<B}
	$\Ee_{H_I} \subseteq \Bb_I$.
\end{enumerate}
\end{lem}
\begin{proof}
(\ref{it:HI}). If $r(\lambda) \in H_I$, then $s_\Ee^c(s(\lambda)) = s_\Ee^c(\lambda)^*
s_\Ee^c(r(\lambda)) s_\Ee^c(\lambda) \in I$, and hence $H_I$ is hereditary. To see that
it is $\Ee$-saturated, suppose that $E \in \Ee$ and $s(E) \subseteq H$, and let $v =
r(E)$. The calculation of \cite[Proposition~8.6]{RaeburnSims:JOT05} implies that
$s_\Ee^c(v) = \Delta(s_\Ee^c)^{\vee E} + \sum_{\mu \in \vee E} s_\Ee^c(\mu)
\Delta(s_\Ee^c)^{T(\vee E;\mu)} s_\Ee^c(\mu)^*$. Since $E \in \Ee$, condition~(S1)
implies that $\vee E \in \Ee$, and so $\Delta(s_\Ee^c)^{ } = 0$. Since $H$ is hereditary,
each $s_\Ee^c(\mu) \Delta(s_\Ee^c)^{T(\vee E; \mu)} s_\Ee^c(\mu)^* \le s_\Ee^c(\mu)
s_\Ee^c(s(\mu))s_\Ee^{c}(\mu)^* \in I$, giving $s_\Ee^c(v) \in I$.

(\ref{it:BI}). Write $\Gamma := \Lambda \setminus \Lambda H$, and fix $F \in \Bb_I$, so
that $\Delta(s_\Ee^c)^F \in I$. To see that $\Bb_I$ satisfies~(S1), fix $\lambda \in
r(F)\Gamma \setminus \{r(F)\}$. Then $\Delta(s_\Ee^c)^{F \cup \{\lambda\}} =
\Delta(s_\Ee^c)^F\big(p_\Ee^c(r(F)) - p_\Ee^c(\lambda)\big) \in I$. For~(S2), fix $F \in
\Bb_I$ and $\lambda \in r(F)\Gamma \setminus F\Gamma$. Then Lemma~\ref{lem:Delta
commutation} gives
\[
\Delta(s_\Ee^c)^{\Ext(\lambda;F)}
	= s_\Ee^c(\lambda)^* \Delta(s_\Ee^c)^F s_\Ee^c(\lambda) \in I.
\]
For~(S3), suppose $\lambda, \lambda\lambda' \in F$ with $\lambda' \neq s(\lambda)$. Then
$p_\Ee^c(r(\lambda))-p_\Ee^c(\lambda\lambda') \geq p_\Ee^c(r(F))-p_\Ee^c(\lambda)\geq
\Delta(s_\Ee^c)^{F\setminus \lambda\lambda'}$, so $\Delta(s_\Ee^c)^{F\setminus
\lambda\lambda'}=\Delta(s_\Ee^c)^{F\setminus\lambda\lambda'}(p_\Ee^c(r(F))-p_\Ee^c(\lambda\lambda'))
= \Delta(s_\Ee^c)^{F} \in I$. For~(S4), suppose that $\lambda \in F$ and $G \in
s(\lambda)\Bb_I$. Let $F' := F \setminus \{\lambda\} \cup \lambda G$. For $\alpha \in G$,
both $p_\Ee^c(r(F))$ and $p_\Ee^c(\lambda)$ dominate $(p_\Ee^c(\lambda) -
p_\Ee^c(\lambda\alpha))$, giving
$(p_\Ee^c(r(F))-p_\Ee^c(\lambda))(p_\Ee^c(\lambda)-p_\Ee^c(\lambda\alpha))=0$. Hence
\begin{align*}
\prod_{\alpha \in G}(p_\Ee^c(r(F)) &{}- p_\Ee^c(\lambda\alpha))\\
	&= \prod_{\alpha \in G}\Big(\big(p_\Ee^c(r(F)) - p_\Ee^c(\lambda)\big) + \big(p_\Ee^c(\lambda) - p_\Ee^c(\lambda\alpha)\big)\Big) \\
	&= \sum_{H \subseteq G}\Big(p_\Ee^c(r(F)) - p_\Ee^c(\lambda)\Big)^{|H|} \prod_{\alpha \in G\setminus H} \Big(p_\Ee^c(\lambda) - p_\Ee^c(\lambda\alpha)\Big) \\
	&= (p_\Ee^c(r(F)) - p_\Ee^c(\lambda)) + \prod_{\alpha \in G} (p_\Ee^c(\lambda) - p_\Ee^c(\lambda\alpha)).
\end{align*}
Hence
\[
\Delta(s_\Ee^c)^{F'}
	= \Delta(s_\Ee^c)^{F \setminus\{\lambda\}}(p_\Ee^c(r(F)) - p_\Ee^c(\lambda))
	+ \Delta(s_\Ee^c)^{F \setminus\{\lambda\}} \prod_{\alpha \in G} (p_\Ee^c(\lambda) -
p_\Ee^c(\lambda\alpha)).
\]
We have $\Delta(s_\Ee^c)^{F \setminus\{\lambda\}}(p_\Ee^c(r(F)) - p_\Ee^c(\lambda)) =
\Delta(s_\Ee^c)^F \in I$, and a quick computation using that range projections commute
shows that $\prod_{\alpha \in G} (p_\Ee^c(\lambda) - p_\Ee^c(\lambda\alpha)) =
s_\Ee^c(\lambda) \Delta(s_\Ee^c)^{G} s_\Ee^c(\lambda)^* \in I$. Hence
$\Delta(s_\Ee^c)^{F'} \in I$.

(\ref{it:E<B}). Fix $F \in \Ee_{H_I}$, and choose $E \in \Ee$ with $r(E) \not\in H$ such
that $F = E \setminus EH$. Then $\Delta(s_\Ee^c)^E = 0$. Let $q_I$ be the quotient map $a
\mapsto a + I$. For $\mu \in E H$, we have $q_I(p_\Ee^c(\mu)) = 0$, and so
$q_I(p_\Ee^c(r(E)) - p_\Ee^c(\mu)) = q_I(p_\Ee^c(r(E))$. Hence
\begin{align*}
0 &= q_I\big(\Delta(s_\Ee^c)^E\big)
	= \prod_{\mu \in E} \big(q_I(p_\Ee^c(r(E)) - p_\Ee^c(\mu))\big) \\
	&= \prod_{\mu \in F} \big(q_I(p_\Ee^c(r(E)) - p_\Ee^c(\mu))\big)
		\prod_{\mu \in EH} q_I(p_\Ee^c(r(E))
	= q_I(\Delta(s_\Ee^c)^F).
\end{align*}
So $\Delta(s_\Ee^c)^F \in I$ and hence $F \in \Bb_I$.
\end{proof}

Suppose that $\Lambda$ is a finitely aligned $k$-graph, that $c \in \Zcat2(\Lambda, \TT)$
and that $\Ee \subseteq \FE(\Lambda)$ is satiated. Suppose that $H \subseteq \Lambda^0$
is hereditary and $\Ee$-saturated and that $\Bb \subseteq \FE(\Lambda \setminus \Lambda
H)$ is satiated and contains $\Ee_H$. We write $I^c_{H, \Bb}$ for the ideal of
$C^*(\Lambda, c; \Ee)$ generated by $\{s_\Ee^c(v) : v \in H\} \cup \{\Delta(s_\Ee^c)^E :
E \in \Bb\}$.

Observe that the cocycle $c$ restricts to a cocycle, which we also denote $c$, on the
subgraph $\Lambda \setminus \Lambda H$.

\begin{thm}\label{thm:quotient map}
Suppose that $\Lambda$ is a finitely aligned $k$-graph, that $c \in \Zcat2(\Lambda, \TT)$
and that $\Ee \subseteq \FE(\Lambda)$ is satiated. Suppose that $H \subseteq \Lambda^0$
is hereditary and $\Ee$-saturated and that $\Bb \subseteq \FE(\Lambda \setminus \Lambda
H)$ is satiated and contains $\Ee_H$. There is a homomorphism $\pi_{H, \Bb} :
C^*(\Lambda, c; \Ee) \to C^*(\Lambda \setminus \Lambda H, c; \Bb)$ such that
\begin{equation}\label{eq:quotient map}
\pi_{H, \Bb}(s_\Ee^c(\lambda))
	= \begin{cases}
		s_{\Bb}^c(\lambda) &\text{ if $s(\lambda) \not\in H$}\\
		0 &\text{ if $s(\lambda) \in H$}.
	\end{cases}
\end{equation}
We have $\ker(\pi_{H, \Bb}) = I_{H, \Bb}$, and $H = H_{I_{H, \Bb}}$ and $\Bb = \Bb_{I_{H,
\Bb}}$.
\end{thm}
\begin{proof}
Routine calculations show that the formula for $\pi_{H, \Bb}$ defines a
Toeplitz-Cuntz-Krieger $(\Lambda, c)$-family in $C^*(\Lambda \setminus \Lambda H, c;
\Bb)$. That $\Ee_H \subseteq \Bb$ ensures that this family also satisfies~(CK). Hence the
universal property of $C^*(\Lambda, c; \Ee)$ gives a homomorphism $\pi_{H, \Bb}$
satisfying \eqref{eq:quotient map}, which is surjective because its image contains the
generators of $C^*(\Lambda \setminus \Lambda H, c; \Bb)$.

To see that $\ker(\pi_{H, \Bb}) = I_{H, \Bb}$, observe that the generators of $I_{H,
\Bb}$ belong to $\ker(\pi_{H, \Bb})$, and so $\pi_{H, \Bb}$ descends to a homomorphism
$\tilde\pi : C^*(\Lambda, c; \Ee)/I_{H, \Bb} \to C^*(\Lambda \setminus \Lambda H, c;
\Bb)$. It suffices to show that $\tilde\pi$ is injective, and we do this by constructing
an inverse for $\tilde\pi$. Define elements $\{t_\lambda : \lambda \in \Lambda \setminus
\Lambda H\}$ in $C^*(\Lambda, c; \Ee)/I_{H, \Bb}$ by $t_\lambda = s_\Ee^c(\lambda) +
I_{H, \Bb}$. Since the relations (TCK1)--(TCK3) for $\Lambda \setminus \Lambda H$
families hold in any Toeplitz-Cuntz-Krieger $\Lambda$ family, $t$ satisfies
(TCK1)--(TCK3). For $\lambda,\mu \in \Lambda \setminus \Lambda H$, it is straightforward
to check that $\MCE_{\Lambda \setminus \Lambda H}(\lambda,\mu) =
\MCE_\Lambda(\lambda,\mu) \setminus \Lambda H$. If $\eta \in \MCE_\Lambda(\lambda,\mu)
\cap \Lambda H$, then $q_\eta$ is the zero element of $C^*(\Lambda, c; \Ee)/I_{H, \Bb}$,
and it follows that the $t_\lambda$ satisfy~(TCK4). They satisfy~(CK) because $E \in \Bb$
implies $\Delta(s_{\Ee}^c)^E \in I_{H, B}$ so that $\Delta(t)^E = 0$. Now the universal
property of $C^*(\Lambda \setminus \Lambda H, c; \Bb)$ provides a homomorphism $\pi_t :
C^*(\Lambda \setminus \Lambda H, c; \Bb) \to C^*(\Lambda, c; \Ee)/I_{H, \Bb}$. The map
$\pi_t$ is an inverse for $\tilde\pi$ on generators, and so $\tilde\pi$ is injective.

For the last assertion, observe that we have $H \subseteq H_{I_{H,\Bb}}$ and $\Bb
\subseteq \Bb_{I_{H, \Bb}}$ by definition. For the reverse inclusions, observe that if $v
\not\in H$ then Proposition~\ref{prp:nonzero elements} applied to $C^*(\Lambda \setminus
\Lambda H, c; \Bb)$ implies that $\pi_{H, \Bb}(s_\Ee^c(v)) \not= 0$, so that $v \not\in
H_{I_{H,\Bb}}$. Similarly if $E \in \FE(\Lambda \setminus \Lambda H) \setminus \Bb$, then
Proposition~\ref{prp:nonzero elements} implies that $\pi_{H, \Bb}(\Delta(s_\Ee^c)^E)
\not= 0$ and hence $E \not\in \Bb_{I_{H, \Bb}}$.
\end{proof}

\begin{cor}\label{cor:quotient iso}
With the hypotheses of Theorem~\ref{thm:quotient map}, the homomorphism $\pi_{H, \Bb}$
descends to an isomorphism $\phi_{H, \Bb} : C^*(\Lambda, c; \Ee)/I_{H, \Bb} \to
C^*(\Lambda \setminus \Lambda H, c; \Bb)$ such that
\[
\phi_{H, \Bb}(s_\Ee^c(\lambda) + I_{H, \Bb}) = s_{\Bb}^c(\lambda) \text{ for all $\lambda
\in \Lambda\setminus\Lambda H$.}
\]
\end{cor}

We are now ready to give a listing of the gauge-invariant ideals of $C^*(\Lambda, c;
\Ee)$. Let $\Lambda$ be a finitely aligned $k$-graph and $\Ee$ a satiated subset of
$\FE(\Lambda)$. We define $\SHS{\Ee}$ to be the collection of all pairs $(H, \Bb)$ such
that $H \subseteq\Lambda^0$ is hereditary and $\Ee$-saturated, and $\Bb \subseteq
\FE(\Lambda \setminus \Lambda H)$ is satiated and satisfies $\Ee_H \subseteq \Bb$.

\begin{thm}\label{thm:g-i ideals}
Let $\Lambda$ be a finitely aligned $k$-graph, $\Ee$ a satiated subset of $\FE(\Lambda)$,
and $c \in \Zcat2(\Lambda, \TT)$. The map  $(H, \Bb) \mapsto I_{H, \Bb}$ is a bijection
from $\SHS{\Ee}$ to the collection of gauge-invariant ideals of $C^*(\Lambda, c; \Ee)$.
We have $I_{H_1, \Bb_1} \subseteq I_{H_2, \Bb_2}$ if and only if both of the following
hold: $H_1 \subseteq H_2$; and whenever $E \in \Bb_1$ and $r(E) \not\in H_2$, we have $E
\setminus E H_2 \in \Bb_2$.
\end{thm}
\begin{proof}
The final statement of Theorem~\ref{thm:quotient map} implies that $(H, \Bb) \mapsto
I_{H, \Bb}$ is injective. To see that it is surjective, fix a gauge-invariant ideal $I$
of $C^*(\Lambda, c; \Ee)$. We must show that $I = I_{H_I, \Bb_I}$. We clearly have
$I_{H_I, \Bb_I} \subseteq I$, and so the quotient map $q_I : C^*(\Lambda, c; \Ee) \to
C^*(\Lambda, c; \Ee)/I$ determines a homomorphism $\tilde{q}_I :  C^*(\Lambda, c;
\Ee)/I_{H_I, \Bb_I} \to C^*(\Lambda, c; \Ee)/I$. We must show that $\tilde{q}_I$ is
injective. Let $\phi_{H_I, \Bb_I} : C^*(\Lambda, c; \Ee)/I_{H_I, \Bb_I} \to C^*(\Lambda
\setminus \Lambda H_I, c; \Bb_I)$ be the isomorphism of Corollary~\ref{cor:quotient iso},
and let $\theta := \tilde{q}_I \circ \phi_{H_I, \Bb_I}$. Since $I$ is gauge-invariant,
the gauge action on $C^*(\Lambda, c; \Ee)$ descends to an action $\beta$ on $C^*(\Lambda,
c; \Ee)/I$ such that $\beta_z \circ \theta = \theta \circ \gamma_z$ for all $z$. For each
$v \in (\Lambda \setminus \Lambda H_I)^0$ we have $s_\Ee^c(v) \not\in I$ by definition of
$H_I$, and so each $\theta(s_{\Bb_I}^c(v)) \not= 0$. Similarly, if $F \in \FE(\Lambda
\setminus \Lambda H) \setminus \Bb_I$, then $\Delta(s_\Ee^c)^F \not\in I$ by definition
of $\Bb_I$, and so $\theta(\Delta(s_\Ee^c)^F) \not= 0$. So the gauge-invariant uniqueness
theorem (Theorem~\ref{thm:giut}) implies that $\theta$ is injective. Hence $\tilde{q}_I$
is injective.

The proof of the final statement is identical to that of \cite[Theorem~6.2]{Sims:CJM06}.
\end{proof}

We are now ready to prove the $\supseteq$ containment from Proposition~\ref{prp:->0 on
filters} as promised in Section 3.

\begin{proof}[Proof of $\supseteq$ in Proposition~\ref{prp:->0 on filters}]
Let
\[
J := \{a \in \Tt C^*(\Lambda, c) : \lim_{\lambda \in S} \pi_T(a) h_\lambda =
    0\text{ for every $\overline{\Ee}$-compatible filter $S$}\}.
\]
This $J$ is clearly a linear subspace of $\Tt C^*(\Lambda)$, and an
$\varepsilon/3$-argument shows that it is norm-closed. If $S$ is an
$\overline{\Ee}$-compatible filter and $s(\lambda) = r(S)$, then $\lambda S$ is a cofinal
subset of the $\overline{\Ee}$-compatible filter $\ell_\lambda(S)$. So if $a \in J$, then
$\lim_{\mu \in S} \pi_T(a s_\lambda) h_\mu = \lim_{\nu \in \lambda S} \pi_T(a)h_{\nu} =
0$, giving $a T_\lambda \in J$. Similarly, since $\ell^*_\lambda(S)$ is an
$\Ee$-compatible filter whenever $S$ is, for $a \in J$ and $\lambda \in \Lambda$ we have
$\lim_{\mu \in S} \pi_T(a s^*_\lambda) h_\mu = \lim_{\nu \in \ell^*_\lambda(S)}
\pi_T(a)h_{\nu} = 0$, so that $a s^*_\lambda \in J$. Clearly $T_\lambda a$ and
$T^*_\lambda a$ belong to $J$. So $J$ is an ideal of $\Tt C^*(\Lambda)$.
Proposition~\ref{prp:->0 on filters} implies that $\Delta(s^c_\Tt)^E \in J$ for all $E
\in \Ee$. Lemma~\ref{lem:nonzero under piT} implies that $s^c_\Tt(v) \not\in \pi_T(J)$
for all $v \in \Lambda^0$ and that $\Delta(s^c_\Tt)^E \not\in J$ for $E \in \FE(\Lambda)
\setminus \Ee$. So $H_J = \emptyset = H_{J_\Ee}$, and $\Bb_J = \Ee = \Bb_{J_\Ee}$. So the
result will follow from Theorem~\ref{thm:g-i ideals} once we establish that $J$ is
gauge-invariant.

Fix $a \in J$ and $z \in \TT^k$; we must show that $\gamma_z(a) \in J$. There is a
unitary $U_z$ in $\Bb(\ell^2(\Lambda))$ given by $U_z(h_\lambda) = z^{d(\lambda)}
h_\lambda$. We have $U_z T_\lambda U^*_z = z^{d(\lambda)} T_\lambda$ for all $z,
\lambda$, and so $\Ad U_z \circ \pi_T = \pi_T \circ \gamma_z$. Fix an
$\overline{\Ee}$-compatible filter $S$. For $\lambda \in S$, we have
\[
\|\pi(\gamma_z(a)) h_\lambda\|
    = \|U_z \pi(a) U^*_z h_\lambda\|
    \le \|U_z\| \|\pi(a) \overline{z}^{d(\lambda)}h_\lambda\|
    = \|\pi(a) h_\lambda\|.
\]
So $\lim_{\lambda \in S} \|\pi(\gamma_z(a)) h_\lambda\| \le \lim_{\lambda \in S} \|\pi(a)
h_\lambda\| = 0$, giving $\gamma_z(a) \in J$.
\end{proof}

\section{Nuclearity and the Universal Coefficient Theorem}\label{sec:nuclear}

We show that each $C^*(\Lambda, c; \Ee)$ is nuclear and satisfies the Universal
Coefficient Theorem. Given a set $X$ we write $\Kk_X$ for the $C^*$-algebra generated by
nonzero matrix units $\{\theta_{x,y} : x,y \in X\}$.

\begin{lem}\label{lem:skew-prod AF}
Let $\Lambda$ be a finitely aligned $k$-graph and let $\Ee$ be a satiated subset of
$\FE(\Lambda)$. Suppose that $b : \Lambda^0 \to \ZZ^k$ satisfies $d(\lambda) =
b(s(\lambda)) - b(r(\lambda))$ for all $\lambda \in \Lambda$. Let $c \in \Zcat2(\Lambda,
\TT)$ and let $t$ be a relative Cuntz-Krieger $(\Lambda, c; \Ee)$-family with each
$t_v\neq 0$.
\begin{enumerate}
\item\label{it:K}
	For $n \in \NN^k$ there is an isomorphism $\clsp\{t_\mu t^*_\nu : b(s(\mu)) =
b(s(\nu)) = n\} \cong \bigoplus_{b(v) = n} \Kk_{\Lambda v}$ which carries each $t_\mu
t^*_\nu$ to $\theta_{\mu,\nu}$.
\item\label{it:absorbtion}
	If $s(\mu) = s(\nu)$ and $s(\eta) = s(\zeta)$, then $t_\mu t^*_\nu t_\eta t^*_\zeta
\in \clsp\{t_\alpha t^*_\beta : b(s(\alpha)) = b(s(\beta)) = b(s(\mu)) \vee b(s(\nu))\}$.
\item\label{it:AF}
	If $N \subseteq \NN^k$ is finite and $m,n \in N$ implies $m \vee n \in N$, then
$B(t)_N := \clsp\{t_\mu t^*_\nu : b(s(\mu)) = b(s(\nu)) \in N\}$ is an AF algebra.
\end{enumerate}
\end{lem}
\begin{proof}
\eqref{it:K}. We just have to check that the $t_\mu t^*_\nu$ are matrix units. So suppose
that $s(\mu) = s(\nu) = v$ and $s(\eta) = s(\zeta) = w$ and $b(v) = b(w) = n$. Then
\[
t_\mu t^*_\nu t_\eta t^*_\zeta
	=  \sum_{\nu\alpha = \eta\zeta \in \MCE(\nu,\eta)}
		c(\mu,\alpha)\overline{c(\nu,\alpha)} c(\zeta,\beta)\overline{c(\eta,\beta)}
		t_{\mu\alpha} t^*_{\zeta\beta}.
\]
If $r(\nu) \not= r(\eta)$, then $\MCE(\nu,\eta) = \emptyset$ and so $t_\mu t^*_\nu t_\eta
t^*_\zeta = 0$. If $r(\nu) = r(\eta)$, then we have $d(\nu) = b(s(\nu)) - b(r(\nu)) = n -
b(r(\nu))$ and similarly $d(\eta) = n - b(r(\eta))$ so that $d(\nu) = d(\eta)$, giving
$t_\mu t^*_\nu t_\eta t^*_\zeta = \delta_{\nu,\eta}t_\mu t^*_\nu t_\nu t^*_\zeta =
\delta_{\nu,\eta}t_\mu t^*_\zeta$.

\eqref{it:absorbtion}. Suppose that $b(s(\nu)) = m$ and $b(s(\eta)) = n$, and that
$\lambda \in \MCE(\nu,\eta)$. Since $\MCE(\nu,\eta) \not= \emptyset$, we have $r(\nu) =
r(\eta)$, and in particular $b(r(\nu)) = b(r(\eta)) = p$, say. Thus $m - d(\nu) = p = n -
d(\eta)$. We have $d(\lambda) = d(\mu) \vee d(\nu) = (m \vee n) - p$ and so
$b(s(\lambda)) = b(r(\lambda)) + (d(\mu) \vee d(\nu)) = m \vee n$. So the result follows
from Lemma~\ref{lem:TCK consequences}.

\eqref{it:AF}. The proof is by induction on $|N|$. For $|N| = 1$, the result follows from
part~(\ref{it:K}). Fix $N$ with $|N| \ge 2$, and suppose as an inductive hypothesis that
$B(t)_{M}$ is AF whenever $|M| < |N|$. Pick a minimal $n \in N$ and observe that $M := N
\setminus \{n\}$ is closed under $\vee$, so that $B_M$ is AF. Part~\ref{it:absorbtion}
implies that $B_M$ is an ideal of $B_N$ and $B_N/B_M$ is a quotient of $B_{\{n\}}$ and
therefore is AF. Since extensions of AF algebras by AF algebras are AF, the result
follows.
\end{proof}

Recall that if $\Lambda$ is a finitely aligned $k$-graph, then $\Lambda \times_d \ZZ^k$
is the skew-product $k$-graph which is equal as a set to $\Lambda \times \ZZ^k$ and has
structure maps $r(\lambda,n) = (r(\lambda), n)$, $s(\lambda, n) = (s(\lambda), n +
d(\lambda))$, $(\lambda, n)(\mu, n+d(\lambda)) = (\lambda\mu, n)$, and $d(\lambda, n) =
d(\lambda)$. For $c \in \Zcat2(\Lambda, \TT)$, the map $c \times 1 : \big((\lambda,
n),(\mu,n+d(\lambda))\big) \mapsto c(\lambda,\mu)$ belongs to $\Zcat2(\Lambda \times_d
\ZZ^k, \TT)$.

\begin{cor}\label{cor:nuclear}
Let $\Lambda$ be a finitely aligned $k$-graph, $\Ee$ a satiated subset of $\FE(\Lambda)$,
and $c \in \Zcat2(\Lambda, \TT)$. Then $C^*(\Lambda, c; \Ee)$ is Morita equivalent to the
crossed-product of an AF algebra by $\ZZ^k$. In particular, it belongs to the bootstrap
class $\mathcal{N}$ of \cite{RosenbergSchochet:DMJ87}, and so is nuclear and satisfies
the UCT.
\end{cor}
\begin{proof}
Let $\Gamma := \Lambda \times_d \ZZ^k$, let $c' := c \times 1$ and let $\Ff := \Ee \times
\ZZ^k$. The argument of \cite[Lemma~8.3]{Sims:CJM06} shows that $\Ff \subseteq
\FE(\Gamma)$ and that the crossed-product of $C^*(\Lambda, c; \Ee)$ by the gauge action
is isomorphic to $C^*(\Gamma, c'; \Ff)$. Since the subalgebras $B_N$ of $C^*(\Gamma, c';
\Ff)$ described in Lemma~\ref{lem:skew-prod AF} are AF algebras and satisfy $C^*(\Gamma,
c'; \Ff) = \overline{\bigcup_N B_N}$, and since the class of AF algebras is closed under
direct limits, $C^*(\Gamma, c'; \Ff)$ is AF. Now Takai duality implies that $C^*(\Lambda,
c; \Ee)$ is Morita equivalent to a crossed product of an AF algebra by $\ZZ^k$, and the
result follows.
\end{proof}

\section{K-theory}\label{sec:K-theory}

In this section we follow the program of \cite{KumjianPaskEtAl:JMAA13} to show that if
$c$ has the form $c(\mu,\nu) = e^{i\omega(\mu,\nu)}$ for some $\omega \in \Zcat2(\Lambda,
\RR)$, then the $K$-theory of $C^*(\Lambda, c; \Ee)$ is isomorphic to that of
$C^*(\Lambda; \Ee)$.

\begin{thm}\label{K-theory_cocycle}
Let $\Lambda$ be a finitely aligned $k$-graph and suppose that $\Ee \subseteq
\FE(\Lambda)$ is satiated. Suppose that $\omega \in \Zcat2(\Lambda, \RR)$, and define $c
\in \Zcat2(\Lambda, \TT)$ by $c(\mu,\nu) = e^{i\omega(\mu,\nu)}$. Then $C^*(\Lambda,
c;\Ee) $ is unital if and only if $\Lambda^0$ is finite. There is an isomorphism:
$K_*(C^*(\Lambda, c; \Ee) \cong K_*(C^*(\Lambda; \Ee))$ taking $[s_\Ee^c(v)]$ to
$[s_\Ee(v)]$ for each $v \in \Lambda^0$, and taking $[1_{C^*(\Lambda, c; \Ee)}]$ to
$[1_{C^*(\Lambda; \Ee)}]$ if $\Lambda^0$ is finite.
\end{thm}

The proof of Theorem \ref{K-theory_cocycle} appears at the end of the section; we have to
do some preliminary work first.

Following \cite[Definition~2.1]{KumjianPaskEtAl:JMAA13}, given a finitely aligned
$k$-graph $\Lambda$, a locally compact abelian group $A$ and a cocycle $\omega \in
\Zcat2(\Lambda, A)$, a \emph{Toeplitz $c$-representation of $(\Lambda, A)$} on a
$C^*$-algebra $B$ consists of a map $\phi : \Lambda \to \Mm (B)$ and a homomorphism $\pi
: C^*(A) \to \Mm (B)$ such that $\phi(\lambda)\pi(f) \in B$ for all $\lambda \in \Lambda$
and $f \in C^*(A)$, and such that
\begin{itemize}
\item[(R1)] $\pi(f)\phi(\lambda) = \phi(\lambda)\pi(f)$ for all $\lambda \in \Lambda$
    and $f \in C^*(A)$;
\item[(R2)] $\{\phi(v) : v \in \Lambda^0\}$ is a set of mutually orthogonal
    projections and $\sum_{v \in \Lambda^0} \phi(v) \to 1$ strictly in $\Mm(B)$;
\item[(R3)] $\phi(\lambda)\phi(\mu) = \pi(\omega(\lambda,\mu))\phi(\lambda\mu)$
    whenever $s(\lambda) = r(\mu)$; and
\item[(R4)] the $\phi(\lambda)$ satisfy (TCK3)~and~(TCK4).
\end{itemize}
If the $\phi(\lambda)$ satisfy relation~(CK) with respect to a satiated subset $\Ee$ of
$\FE(\Lambda)$, then $(\phi, \pi)$ is an \emph{$\Ee$-relative $c$-representation} of
$(\Lambda, A)$.

Calculations identical to those of Lemma~\ref{lem:TCK consequences} show that given a
Toeplitz $c$-representation of $(\Lambda, A)$, the $C^*$-algebra $C^*(\phi,\pi) :=
C^*\{\phi(\lambda)\pi(f) : f \in C^*(A), \lambda \in \Lambda\}$ is spanned by
$\{\phi(\mu)\pi(f)\phi(\nu)^* : s(\mu) = s(\nu), f \in C^*(A)\}$, and that
\[
\phi(\nu)^* \phi(\eta)
	= \sum_{\nu\alpha = \eta\beta \in \MCE(\nu,\eta)}
	\phi(\alpha) \pi(\omega(\eta,\beta))\pi(\omega(\nu,\alpha))^* \phi(\beta)^*,
\]
where we are identifying elements of $A$ with the corresponding multiplier unitaries of
$C^*(A)$. There is a universal $C^*$-algebra $C^*(\Lambda, A, \omega; \Ee)$ for
$\Ee$-relative $c$-representations of $(\Lambda, A)$, and we denote the universal
$\Ee$-relative representation by $(\iota^{\Ee, \omega}_\Lambda, \iota^{\Ee, \omega}_A)$.
The argument of \cite[Lemma~2.4]{KumjianPaskEtAl:JMAA13} shows that $C^*(\Lambda, A,
\omega; \Ee)$ is canonically isomorphic to $C^*(\Lambda, A, \omega'; \Ee)$ if $\omega$
and $\omega'$ are cohomologous. Let $\widehat{A}$ denote the Pontryagin dual of $A$. As
in Proposition~2.5 of \cite{KumjianPaskEtAl:JMAA13}, the algebra $C^*(\Lambda, A, \omega;
\Ee)$ is a $C_0(\widehat{A})$-algebra with respect to the inclusion
$i_A^{\Ee,\omega}:C^*(A) \cong C_0(\widehat{A})\hookrightarrow Z\Mm(C^*(\Lambda, A,
\omega; \Ee))$, and for each $\chi \in \widehat{A}$ there is a homomorphism $\pi_\chi :
C^*(\Lambda, A, \omega; \Ee) \to C^*(\Lambda, \chi\circ\omega; \Ee)$ satisfying
$\pi_\chi(\iota^{\Ee, \omega}_\Lambda(\lambda)\iota^{\Ee, \omega}_A(f)) =
f(\chi)s_\Ee^c(\lambda)$. Moreover, $\pi_\chi$ descends to an isomorphism of the fibre
$C^*(\Lambda, A, \omega; \Ee)_\chi$ with $C^*(\Lambda, \chi\circ\omega; \Ee)$.

When $\omega \equiv 1$, the universal properties of $C^*(\Lambda, A, 1, \Ee)$ and
$C^*(\Lambda, 1; \Ee) \otimes C^*(A)$ give an isomorphism $C^*(\Lambda, A, 1; \Ee) \cong
C^*(\Lambda, 1; \Ee) \otimes C^*(A)$.

We now consider a finitely aligned $k$-graph $\Gamma$ endowed with a map $b : \Gamma^0
\to \ZZ^k$ such that $d(\lambda) = b(s(\lambda)) - b(r(\lambda))$ for all $\lambda \in
\Gamma$. Our application is when $\Gamma$ is the skew-product $\Lambda \times_d \ZZ^k$ as
in the preceding section, but it will keep our notation simpler to deal with the general
situation.

The proof of Lemma~8.2 of \cite{Sims:CJM06} shows that given a finite subset $E$ of
$\Gamma$, there is a minimal finite $\tE \subseteq \Gamma$ such that $E \subseteq \tE$
and whenever $\mu,\nu,\eta,\zeta \in \tE$ with $s(\mu)=s(\nu)$, $s(\eta)=s(\zeta)$ and
$\nu\alpha = \eta\beta \in \MCE(\nu,\eta)$, we have $\mu\alpha, \zeta\beta \in \tE$.
(This set plays a similar role to the set $\Pi E$ used to examine the core in
Section~\ref{sec:uniqueness thms}, and is constructed in a similar way.) The set $\tE$
satisfies $\tE = \vee \tE$ and has the property that $T(\tE; \mu) = T(\tE; \nu)$ whenever
$\mu,\nu \in \tE$ and $s(\mu) = s(\nu)$.

Let $\omega \in \Zcat2(\Gamma, A)$ and let $(\phi, \pi)$ be a Toeplitz
$\omega$-representation of $(\Gamma, A)$. For $v \in \Lambda^0$ and a finite $E \subseteq
v\Gamma$, let $\Delta(\phi)^E := \prod_{\lambda \in E} \phi(v) -
\phi(\lambda)\phi(\lambda)^*$. Since the $\pi(\omega(\mu,\nu))$ are unitaries and each
$\pi(\omega(r(\mu),\mu)) = 1_{\Mm(B)}$, the calculations of Lemma~\ref{lem:Delta
commutation} show that
\begin{equation}\label{eq:phi-Delta-comm}
\Delta(\phi)^E \phi(\mu) = \phi(\mu) \Delta(\phi)^{\Ext(\mu;E)}.
\end{equation}
The $\phi(\mu)\phi(\mu)^*$ satisfy the same commutation relations as the $t_\mu t^*_\mu$
in a Toeplitz-Cuntz-Krieger family. So if $E = \tE$, then since $\tE = \vee\tE$ the
argument of \cite[Corollary~3.7]{RaeburnSimsEtAl:JFA04} gives
\begin{equation}\label{eq:projection decomp}
\phi(\lambda) \phi(\lambda)^* = \sum_{\lambda\lambda' \in E}
\Delta(\phi)^{T(E;\lambda\lambda')}.
\end{equation}
The argument of Lemma~\ref{lem:mx units} shows that $\Theta(\phi)^E_{\mu, \nu} :=
\phi(\mu) \Delta(\phi)^{T(E;\mu)} \phi(\nu)^*$ defines matrix units, and that for
$\mu,\nu \in E$ with $s(\mu) = s(\nu)$, we have
\begin{equation}\label{eq:Thetas C*A-span}
\phi(\mu)\phi(\nu)^* = \sum_{\mu\alpha \in E}
\pi(\omega(\mu,\alpha))\pi(\omega(\nu,\alpha))^* \Theta(\phi)^E_{\mu\alpha,\nu\alpha}.
\end{equation}

\begin{lem}\label{lem:mx times A}
Let $\Gamma$ be a finitely aligned $k$-graph and suppose that $b : \Lambda^0 \to \ZZ^k$
satisfies $d(\lambda) = b(s(\lambda)) - b(r(\lambda))$ for all $\lambda$. Suppose that
$\Ff \subseteq \FE(\Gamma)$ is satiated. Let $A$ be a locally compact abelian group and
consider $\omega \in \Zcat2(\Gamma, A)$. If $E = \tE \subseteq \Gamma$ then
\begin{align*}
M^{\Ff, \omega}_E &:= \lsp\{\iota^{\Ff, \omega}_\Gamma(\mu) \iota^{\Ff,\omega}_A(f)
\iota^{\Ff, \omega}_\Gamma(\nu)^* :
		\mu,\nu \in E, f \in C^*(A)\} \\
	&= \lsp\{\Theta(\iota^{\Ff, \omega}_\Gamma)^E_{\mu, \nu}\iota^{\Ff,\omega}_A(f) :
\mu,\nu \in E, f \in C^*(A)\},
\end{align*}
and there is an isomorphism of $M^{\Ff, \omega}_E$ onto $\bigoplus_{v \in s(E)}
M_{Ev}(\CC) \otimes C^*(A)$ that carries each spanning element $\Theta(\iota^{\Ff,
\omega}_\Gamma)^E_{\mu, \nu}\iota^{\Ff,\omega}_A(f)$ to $\theta_{\mu,\nu} \otimes f$.
\end{lem}
\begin{proof}
Equation~\eqref{eq:Thetas C*A-span} establishes the displayed equation. We saw above that
the $\Theta(\iota^{\Ff, \omega}_\Gamma)^E_{\mu, \nu}$ are matrix units, and they commute
with the $\iota^{\Ff,\omega}_A(f)$ because the range of $\iota^{\Ff,\omega}_A$ is
central. Now we follow the argument of Lemma~4.1 of \cite{KumjianPaskEtAl:JMAA13}: The
universal property of $\bigoplus_{v \in s(E)} M_{Ev}(\CC) \otimes C^*(A)$ gives a
surjection $\psi : \bigoplus_{v \in s(E)} M_{Ev}(\CC) \otimes C^*(A) \to M^{\Ff,
\omega}_E$ such that $\psi(\theta_{\mu,\nu} \otimes f) = \Theta(\iota^{\Ff,
\omega}_\Gamma)^E_{\mu, \nu}\iota^{\Ff,\omega}_A(f)$. For each $\chi \in \widehat{A}$,
and each $f \in \widehat{A}$ such that $f(\chi) = 1$, the canonical homomorphism
$\pi_\chi : C^*(\Gamma, A, \omega; \Ff) \to C^*(\Gamma, \chi\circ\omega; \Ee)$ carries
the $\Theta(\iota^{\Ee, \omega}_\Gamma)^E_{\mu, \nu}\iota^{\Ee,\omega}_A(f)$ to the
matrix units $\theta(s_\Ee^{\chi\circ\omega})^E_{\mu,\nu}$, and Lemma~\ref{lem:mx units}
shows that any given $\theta(s_\Ee^{\chi\circ\omega})^E_{\mu,\nu}$ is nonzero if and only
if $\theta(s_\Ee)^E_{\mu,\nu}$ is nonzero; so $\pi_\chi$ determines an isomorphism
$(M^{\Ee, \omega}_E\big)_\chi \cong \bigoplus_{v \in s(E)} M_{Ev}(\CC)$. So $\psi$
descends to an isomorphism of each fibre in the trivial bundle $\bigoplus_{v \in s(E)}
M_{Ev}(\CC) \otimes C^*(A)$, and so is isometric by
\cite[Proposition~C.10(c)]{Williams:CP}.
\end{proof}

\begin{lem}\label{lem:inclusions}
Let $\Gamma$ be a finitely aligned $k$-graph. Let $A$ be a locally compact abelian group
and consider $\omega \in \Zcat2(\Lambda, A)$. Let $(\phi, \pi)$ be a Toeplitz
$\omega$-representation of $(\Gamma, A)$. Suppose that $E \subseteq F$ are finite subsets
of $\Gamma$ satisfying $E = \tE$ and $F = \tF$. For each $\mu \in F$ there is a unique
maximal $\iota_\mu \in E$ such that $\mu  \in \iota_\mu\Gamma$, and we have
\[
\Theta(\phi)^E_{\mu,\nu} = \sum_{\mu\alpha \in F, \iota_{\mu\alpha} = \mu}
	\pi(\omega(\mu,\alpha))\pi(\omega(\nu,\alpha))^*
\Theta(\phi)^F_{\mu\alpha,\nu\alpha}.
\]
\end{lem}
\begin{proof}
Let $N=\bigvee\{d(\lambda) : \lambda \in E, \mu \in \lambda \Gamma\}$. Then $N \leq
d(\mu)$ and since $E = \vee E$, factorising $\mu=\iota_\mu \mu'$ with $d(\iota_\mu)=N$
gives the desired $\iota_\mu$. We claim that for $\mu \in E$ and $\lambda \in F$ we have
$\Theta(\phi)^E_{\mu,\mu} \Theta(\phi)^F_{\lambda,\lambda} = \delta_{\iota_\lambda, \mu}
\Theta(\phi)^F_{\lambda,\lambda}$. First suppose that $\iota_\lambda = \mu$. We have
\begin{align*}
\Theta(\phi)&^E_{\mu,\mu} \Theta(\phi)^F_{\lambda,\lambda} \\
	&= \phi(\mu)\phi(\mu)^* \prod_{\mu' \in T(E;\mu)} \big(\phi(\mu)\phi(\mu)^* -
\phi(\mu\mu')\phi(\mu\mu')^*\big)
\phi(\lambda)\phi(\lambda)^*\Theta(\phi)^F_{\lambda,\lambda}.
\end{align*}
For $\mu' \in T(E;\mu)$ the maximality of $\iota_\lambda = \mu$ ensures that $\lambda
\not\in \mu\mu'\Lambda$. Since $F = \vee F$ and contains $E$, we deduce that if
$\mu\mu'\alpha = \lambda\beta \in \MCE(\mu\mu',\lambda)$, then
$\phi(\lambda)\phi(\lambda)^* - \phi(\lambda\beta)\phi(\lambda\beta)^*$ is a factor in
$\Theta(\phi)^F_{\lambda,\lambda}$, and so $\phi(\mu\mu')\phi(\mu\mu')^*
\Theta(\phi)^F_{\lambda,\lambda} = 0$. So the preceding displayed equation collapses to
\[
\Theta(\phi)^E_{\mu,\mu} \Theta(\phi)^F_{\lambda,\lambda}
	= \phi(\mu)\phi(\mu)^*\Theta(\phi)^F_{\lambda,\lambda} =
\Theta(\phi)^F_{\lambda,\lambda}.
\]

Now suppose that $\mu \not= \iota_\lambda$. Since $\iota_\lambda$ is the maximal initial
segment of $\lambda$ in $E$, we have two cases to consider: either $\iota_\lambda \in
\mu\Lambda \setminus \{\mu\}$ or $\lambda \not\in \mu\Lambda$.  First suppose that
$\iota_\lambda \in \mu\Lambda \setminus\{\mu\}$. Then $\phi(\mu)\phi(\mu)^* -
\phi(\iota_\lambda)\phi(\iota_\lambda)^* \ge \Theta(\phi)^E_{\mu,\mu}$. Since
$\big(\phi(\mu)\phi(\mu)^* - \phi(\iota_\lambda)\phi(\iota_\lambda)^*\big) \perp
\phi(\lambda)\phi(\lambda)^* \ge \Theta(\phi)^F_{\lambda,\lambda}$ it follows that
$\Theta(\phi)^E_{\mu,\mu} \Theta(\phi)^F_{\lambda,\lambda} = 0$. Now suppose that
$\lambda \not\in \mu\Lambda$. Then $\mu\alpha = \lambda\beta \in \MCE(\mu,\lambda)$
implies $\beta \in  T(F;\lambda)$, and then the argument of the preceding paragraph gives
$\phi(\mu)\phi(\mu)^*\Theta(\phi)^F_{\lambda,\lambda} = 0$. Hence
$\Theta(\phi)^E_{\mu,\mu} \Theta(\phi)^F_{\lambda,\lambda} = 0$.  This proves the claim.

Now fix $\mu,\nu \in E$. Equations \eqref{eq:phi-Delta-comm}~and~\eqref{eq:projection
decomp} imply that
\[
\Theta(\phi)^E_{\mu,\nu}
	= \sum_{\mu\alpha \in F} \Theta(\phi)^F_{\mu\alpha,\mu\alpha}
\Theta(\phi)^E_{\mu,\mu}\phi(\mu)\phi(\nu)^*,
\]
and the claim reduces this to
\[
\sum_{\substack{\mu\alpha \in F\\ \iota_{\mu\alpha} = \mu}}
\Theta(\phi)^F_{\mu\alpha,\mu\alpha}\phi(\mu)\phi(\nu)^*
	= \sum_{\substack{\mu\alpha \in F\\ \iota_{\mu\alpha} = \mu}}
\Theta(\phi)^F_{\mu\alpha,\mu\alpha}
		\phi(\mu\alpha)\phi(\mu\alpha)^*\phi(\mu)\phi(\nu)^*.
\]
A by-now familiar computation using the cocycle identity transforms this into
\[
\sum_{\mu\alpha \in F, \iota_{\mu\alpha} = \mu}
\pi(\omega(\mu,\alpha))\pi(\omega(\nu,\alpha))^* \Theta(\phi)^F_{\mu\alpha,\mu\alpha}
\phi(\mu\alpha)\phi(\nu\alpha)^*,
\]
and another application of~\eqref{eq:phi-Delta-comm} completes the proof.
\end{proof}

Now given $\Gamma, \Ff$ and $\omega$ as above and a finite subset $E = \tE$ of $\Gamma$,
we define a map  $\iota^E : \Gamma \to \Gamma$ as follows: if $\lambda \in E\Gamma$, then
$\iota^E_\lambda$ is the maximal element of $E$ such that $\lambda \in \iota^E_\lambda
\Gamma$. For all other $\lambda \in \Gamma$ we set $\iota^E_\lambda = \lambda$. We define
$\tau^E : \Gamma \to \Gamma$ by $\lambda = \iota_\lambda \tau_\lambda$ for all $\lambda$.

\begin{thm}\label{thm:AF field}
Let $\Gamma$ be a finitely aligned $k$-graph. Suppose that there is a map $b : \Gamma^0
\to \ZZ^k$ such that $d(\lambda) = b(s(\lambda)) - b(r(\lambda))$ for all $\lambda$ and
that $\Ff \subseteq \FE(\Gamma)$ is satiated. Let $A$ be a locally compact abelian group
and consider $\omega \in \Zcat2(\Lambda, A)$. There is an isomorphism $C^*(\Lambda, A,
\omega; \Ff) \cong C^*(\Lambda; \Ff) \otimes C^*(A)$ which carries $\iota^{\Ff,
\omega}_\Gamma(\lambda) \iota^{\Ff, \omega}_A(f) \iota^{\Ff, \omega}_\Gamma(\lambda)^*$
to $s_\Ff(\lambda) s_\Ff(\lambda)^* \otimes f$ for all $\lambda \in \Lambda$.
\end{thm}
\begin{proof}
Fix an increasing sequence $E_1 \subseteq E_2 \subseteq E_3 \subseteq  \cdots$ of subsets
of $\Gamma$ such that each $E_i = \tE_i$ and $\bigcup_i E_i = \Gamma$. Recursively define
maps $\kappa_i : \{(\mu,\nu) \in \Gamma : s(\mu) = s(\nu)\} \to \Mm C^*(A)$ by $\kappa_0
\equiv 1_{\Mm C^*(A)}$ and
\[
\kappa_{i+1}(\mu,\nu) = \kappa_i(\mu,\nu) - \omega(\iota^{E_i}_\mu, \tau^{E_i}_\mu)
    + \omega(\iota^{E_i}_\nu, \tau^{E_i}_\nu).
\]

For each $i$, define a linear map $\psi_i : M^{\Ff, \omega}_{E_i} \to M^\Ff_{E_i}$ by
\[
\psi_i\big(\iota^{\Ff, \omega}_A(f)\Theta(\iota^{\Ff,
\omega}_\Gamma)^{E_i}_{\mu,\nu}\big) = \iota^{\Ff,
\omega}_A(\kappa_i(\mu,\nu))\iota^{\Ff,
\omega}_A(f)\Theta(\iota^{\Ff}_\Gamma)^{E_i}_{\mu,\nu}.
\]
An induction argument shows that each $\kappa_i$ satisfies $\kappa_i(\nu,\mu) =
-\kappa_i(\mu,\nu)$ and that $\kappa_i(\lambda,\mu)+\kappa_i(\mu,\nu) =
\kappa_i(\lambda,\nu)$, and so $\psi_i$ preserves adjoints and multiplication. Two
applications of Lemma~\ref{lem:mx times A} show that $\psi_i$ is an isomorphism. Using
the definition of the $\kappa_i$ and Lemma~\ref{lem:inclusions}, we calculate:
\begin{align*}
\psi_i&(\Theta(\iota^{\Ff, \omega}_\Gamma)_{\mu,\nu}) \\
	&= \iota^{\Ff, \omega}_A(\kappa_i(\mu,\nu)) \Theta(\iota^{\Ff}_\Gamma)^{E_i}_{\mu,\nu} \\
	&= \sum_{\mu\alpha \in E_{i+1}, \iota_{\mu\alpha} = \mu}
		\iota^{\Ff, \omega}_A(\kappa_i(\mu,\nu)) \Theta(\iota^{\Ff}_\Gamma)^{E_{i+1}}_{\mu\alpha,\nu\alpha} \\
	&=  \sum_{\mu\alpha \in E_{i+1}, \iota_{\mu\alpha} = \mu}
		\iota^{\Ff, \omega}_A(\kappa_{i+1}(\mu\alpha,\nu\alpha) + \omega(\mu,\alpha) -
\omega(\nu,\alpha))
		\Theta(\iota^{\Ff}_\Gamma)^{E_{i+1}}_{\mu\alpha,\nu\alpha} \\
	&= \sum_{\mu\alpha \in E_{i+1}, \iota_{\mu\alpha} = \mu}
		\psi_{i+1}\big(\iota^{\Ff, \omega}_A(\omega(\mu,\alpha))\iota^{\Ff,
\omega}_A(\omega(\nu,\alpha))^*
		\Theta(\iota^{\Ff, \omega}_\Gamma)^{E_{i+1}}_{\mu\alpha,\nu\alpha}\big)\\
	&= \psi_{i+1}(\Theta(\iota^{\Ff, \omega}_\Gamma)^{E_i}_{\mu,\nu}).
\end{align*}
Hence there is an isomorphism $\psi_\infty : C^*(\Gamma, A, \omega; \Ff) \to C^*(\Gamma,
A; \Ff)$ such that $\psi_\infty|_{M^{\Ff, \omega}_{E_i}} = \psi_i$. The
formula~\eqref{eq:projection decomp} gives $\psi_\infty\big(\iota^{\Ff,
\omega}_\Gamma(\lambda) \iota^{\Ff, \omega}_A(f) \iota^{\Ff,
\omega}_\Gamma(\lambda)^*\big) = s_\Ff(\lambda) s_\Ff(\lambda)^* \otimes f$ for each
$\lambda$ and $f$.
\end{proof}

If $A$ is a $C_0(X)$-algebra in the sense that there is a nondegenerate homomorphism $i :
C_0(X) \to \Zz\Mm(A)$, and if $I \subseteq X$ is closed, then we write $A|_I$ for the
quotient of $A$ by the ideal generated by $\{i(f) : f|_I = 0\}$.

\begin{cor}\label{cor:core eval iso}
Let $\Lambda$ be a finitely aligned $k$-graph and suppose that $\Ee \subseteq
\FE(\Lambda)$ is satiated. Suppose that $\omega \in \Zcat2(\Lambda, \RR)$. There is an
isomorphism $C^*(\Lambda \times_d \ZZ^k, A, \omega \times 1; \Ee \times \ZZ^k)|_{[0,1]}
\cong C^*(\Lambda \times_d \ZZ^k; \Ee \times \ZZ^k) \otimes C([0,1])$ that carries
$\iota^{\Ee \times \ZZ^k\!,\, \omega \times 1}_{\Lambda \times_d \ZZ^k}(\lambda)
\iota^{\Ee \times \ZZ^k\!,\, \omega \times 1}_{A}(f)\iota^{\Ee \times \ZZ^k\!,\, \omega
\times 1}_{\Lambda \times_d \ZZ^k}(\lambda)^*$ to $s_{\Ee \times
\ZZ^k}^{\chi\circ\omega}(\lambda) s_{\Ee \times \ZZ^k}^{\chi\circ\omega}(\lambda)^*
\otimes f$ for all $\lambda, f$. In particular, evaluation at any character $\chi \in
[0,1] \subseteq \widehat{\RR}$ determines an isomorphism $K_*(C^*(\Lambda \times_d \ZZ^k,
A, \omega \times 1; \Ee \times \ZZ^k)|_{[0,1]}) \cong K_*(C^*(\Lambda \times_d \ZZ^k,
\chi\circ\omega; \Ee \times \ZZ^k))$ which carries $\Big[\iota^{\Ee \times \ZZ^k\!,\,
\omega \times 1}_{\Lambda \times_d \ZZ^k}(\lambda) \iota^{\Ee \times \ZZ^k\!,\, \omega
\times 1}_{\Lambda \times_d \ZZ^k}(\lambda)^*\Big]$ to $\Big[s_{\Ee \times
\ZZ^k}^{\chi\circ\omega}(\lambda) s_{\Ee \times
\ZZ^k}^{\chi\circ\omega}(\lambda)^*\Big]$.
\end{cor}
\begin{proof}
Use Theorem~\ref{thm:AF field} and the K\"unneth theorem (see
\cite[Corollary~4.3]{KumjianPaskEtAl:JMAA13}).
\end{proof}

We can now prove the main result of the section, Theorem \ref{K-theory_cocycle}

\begin{proof}[Proof of Theorem \ref{K-theory_cocycle}]
We follow the proof of \cite[Theorem~5.4]{KumjianPaskEtAl:JMAA13}. As in
Corollary~\ref{cor:nuclear}, the argument of \cite[Lemma~8.3]{Sims:CJM06} shows that
$C^*(\Lambda \times_d \ZZ^k, \RR, \omega \times 1; \Ee \times \ZZ^k)$ is isomorphic to
the crossed product of $C^*(\Lambda, \RR, \omega; \Ee)$ by the gauge action and that the
inclusion of $C^*(\Lambda, \RR, \omega; \Ee)$ into the crossed product takes the $K_0$
class of an $\iota^{\Ee, \omega}_\Lambda(v)$ to the class of $\iota^{\Ee \times \ZZ^k;
\omega \times 1}_{\Lambda \times_d \ZZ^k}((v,0))$.

The argument of \cite[Lemma~5.2]{KumjianPaskEtAl:JMAA13} goes through more or less
verbatim (substitute ``relative Cuntz-Krieger family" for ``Cuntz-Krieger family" as
necessary) to prove that translation in the $\ZZ^k$ coordinate on $\Lambda \times_d
\ZZ^k$ induces the action $\hat{\gamma}$ of $\ZZ^k$ on $C^*(\Lambda \times_d \ZZ^k, \RR,
\omega; \Ee)$ that is dual to the gauge-action, that the projection $P_0 = \sum_{v \in
\Lambda^0} \iota^{\Ee \times \ZZ^k; \omega \times 1}_{\Lambda \times_d \ZZ^k}((v,0))$ is
full in $C^*(\Lambda \times_d \ZZ^k, \RR, \omega \times 1; \Ee \times \ZZ^k)
\times_{\hat\gamma} \ZZ^k$ and that the corner it determines is isomorphic to
$C^*(\Lambda, \RR, \omega; \Ee)$ via an isomorphism that takes the generating projection
associated to $(v,0) \in (\Lambda \times_d \ZZ^k)^0$ to the generating projection
associated to $v \in \Lambda$.

Lemma~5.3 of \cite{KumjianPaskEtAl:JMAA13} implies that the inclusion of $C^*(\RR)$ in
the centre of $C^*(\Lambda \times_d \ZZ^k, \RR, \omega \times 1; \Ee \times \ZZ^k)
\times_{\hat{\gamma}} \ZZ^k$ makes it into a $C_0(\RR)$ algebra whose fibre over $\chi
\in \widehat{\RR}$ is $C^*(\Lambda \times_d \ZZ^k, \chi\circ(\omega \times 1); \Ee \times
\ZZ^k) \times_{\hat{\gamma}^{\chi\circ\omega}} \ZZ^k$. Since Corollary~\ref{cor:core eval
iso} implies that evaluation at each point of $[0,1]$ induces an isomorphism in
$K$-theory on $C^*(\Lambda \times_d \ZZ^k, \RR, \omega \times 1; \Ee \times
\ZZ^k)|_{[0,1]}$, Theorem~5.1 of \cite{KumjianPaskEtAl:JMAA13} implies that the same is
true of the crossed product; applying this at $t=0$ and $t=1$ gives the result.
\end{proof}

\section{Simplicity}\label{sec:simplicity}

In Corollary~\ref{cor:CKUT} we showed that if $\Lambda$ is cofinal and aperiodic in the
sense of \cite{LewinSims:MPCPS10}, then each $C^*(\Lambda, c)$ is simple. In the
untwisted setting, these conditions are also necessary (see
\cite[Theorem~3.4]{LewinSims:MPCPS10}), but the example of rotation algebras, discussed
in the final paragraph of the introduction, shows that this is not to be expected in the
twisted setting. In this section we give a sufficient condition for simplicity of twisted
$C^*$-algebras associated to $k$-graphs that are not aperiodic.

A \emph{bicharacter} of $\ZZ^k$ is a map $c : \ZZ^k \times \ZZ^k \to \TT$ such that
\[
c(m,n)c(m,n') = c(m,n+n')\qquad\text{and}\qquad c(m,n)c(m',n) = c(m+m', n).
\]
This implies that $c(0,n) = c(m,0) = 1$ and that $c(-m,n) = \overline{c(m,n)} = c(m,-n)$
for all $m,n$. A bicharacter is \emph{skew-symmetric}\footnote{symplectic in the language
of \cite{OPT}.} if it has the additional property that $c(n,m) = \overline{c(m,n)}$ for
all $m,n$.

For $c \in Z^2(\ZZ^k, \TT)$, we write $c^*$ for the cocycle $c^*(m,n) =
\overline{c(n,m)}$ (note the reversal of variables; so $c^* \not= \overline{c}$).
Proposition~3.2 of \cite{OPT} shows that the product $cc^*$, given by $(cc^*)(m,n) =
c(m,n)\overline{c(n,m)}$, is a skew-symmetric bicharacter, and moreover that the map $c
\mapsto cc^*$ has kernel $B^2(\ZZ^k, \TT)$ and determines an isomorphism of $H^2(\ZZ^k,
\TT)$ onto the group of skew-symmetric bicharacters of $\ZZ^k$. It follows immediately
(or see \cite{Brown:cohomology}) that $H^2(\ZZ^k, \TT)$ is isomorphic to
$\TT^{k(k-1)/2}$, and that every class has a representative that is a bicharacter: given
$c \in Z^2(\ZZ^k, \TT)$, form the skew-symmetric bicharacter $cc^*$, and let $\tilde{c}$
be the unique bicharacter of $\ZZ^k$ such that $\tilde{c}(e_i, e_j) = cc^*(e_i, e_j)$ if
$i > j$ and $\tilde{c}(e_i, e_j) = 1$ if $i \le j$. Then $\tilde{c}\tilde{c}^*(e_i, e_j)
= cc^*(e_i, e_j)$ for all $i,j$, and since both are bicharacters, the two agree. Hence
\cite[Proposition~3.2]{OPT} discussed above shows that $c$ and $\tilde{c}$ represent the
same class in $H^2(\ZZ^k, \TT)$.

A skew-symmetric bicharacter $c$ is called \emph{nondegenerate} if there is no nonzero $m
\in \ZZ^k$ such that $c(m,n) = 1$ for all $n$. So if $c \in Z^2(\ZZ^k, \TT)$ is a
bicharacter, then $cc^*$ is nondegenerate if and only if there is no nonzero $m$ such
that $c(n,m) = c(m,n)$ for all $n$ (equivalently, no nonzero $m$ satisfies
$c(m,e_i)=c(e_i,m)$ for all $i \in \{1,\cdots, k\}$).

Given $c \in Z^2(\ZZ^k, \TT)$, the noncommutative torus $A(c)$ is the universal algebra
generated by unitaries $\{U_m : m \in \ZZ^k\}$ satisfying $U_m U_n = c(m,n)U_{m+n}$ for
all $m,n \in \ZZ^k$. An argument like that of Proposition~\ref{prp:H2-dependence} shows
that if $c_1$ and $c_2$ are cohomologous in $Z^2(\ZZ^k, \TT)$, then $A(c_1) \cong
A(c_2)$. So the remarks in the preceding paragraph (see also Remark~1.2 of
\cite{Phillips:xx06}) show that $A(c)$ is universal for unitaries $\{U_i : i \le k\}$
satisfying $U_i U_j = cc^*(e_i, e_j)U_j U_i$.  Theorem~3.7 of \cite{Slawny:CMP72}
combined with~1.8 of \cite{Elliott:MSM80} (see \cite[Theorem~1.9]{Phillips:xx06}) shows
that the noncommutative torus $A(c)$ is simple if and only if $cc^*$ is nondegenerate.

To state the main result of the section, observe that if $\Lambda$ is row-finite and has
no sources, then $v\Lambda^n \in \FE(\Lambda)$ for each $v \in \Lambda^0$ and $n \in
\NN^k\setminus\{0\}$. Hence a filter $S$ is $\FE(\Lambda)$-compatible if and only if it
is maximal in the sense that $S \cap \Lambda^n \not= \emptyset$ for all $n \in \NN^k$. We
call such filters \emph{ultrafilters}. Recall that if $S$ is an ultrafilter, and $s(\mu)
= r(S)$, then $\ell_\mu(S)$ is the filter $\{\nu \in \lambda : \nu\Lambda \cap \mu S
\not= \emptyset\}$, which is then also an ultrafilter.

Under the bijection of Remark~\ref{rmk:filters vs paths}, the ultrafilters of $\Lambda$
correspond to the infinite paths used in \cite{KumjianPask:NYJM00}, and if this bijection
carries the infinite path $x$ to the ultrafilter $S$, then for $\mu \in \Lambda r(x)$, it
carries the infinite path $\mu x$ to the ultrafilter $\ell_\mu(S)$.

As in \cite{CarlsenKangEtAl:xx13}, for $\mu,\nu$ in a row-finite $k$-graph $\Lambda$ with
no sources, we write $\mu \sim \nu$ if $\mu x  = \nu x$ for every infinite path $x$ in
$s(\mu)\Lambda^\infty$. Equivalently, $\mu \sim \nu$ if $\ell_\mu(S) = \ell_\nu(S)$ for
every ultrafilter $S$ such that $r(S) = s(\mu)$. We define $\Per(\Lambda)$ to be the
subgroup of $\ZZ^k$ generated by $\{d(\mu) - d(\nu) : \mu \sim \nu\}$.

In this section, if $c \in Z^2(\ZZ^k, \TT)$ and $\Lambda$ is a $k$-graph with degree
functor $d$, then we abuse notation slightly and write $c \circ d$ for the cocycle
$c\circ d(\lambda,\mu) := c(d(\lambda), d(\mu))$.

\begin{thm}\label{thm:simple}
Let $\Lambda$ be a row-finite $k$-graph with no sources, and take $c \in Z^2(\ZZ^k,
\TT)$. Suppose that $(cc^*)|_{\Per(\Lambda)}$ is nondegenerate. Then $C^*(\Lambda, c
\circ d)$ is simple if and only if $\Lambda$ is cofinal.
\end{thm}

For our first couple of results, we continue to work in the generality of finitely
aligned $k$-graphs. We begin by showing that cofinality of $\Lambda$ is necessary for
simplicity of $C^*(\Lambda, c)$.

\begin{lem}\label{lem:cofinality necessary}
Let $\Lambda$ be a finitely aligned $k$-graph, and suppose that $c \in \Zcat{2}(\Lambda,
\TT)$. Suppose that $\Lambda$ is not cofinal. Then $C^*(\Lambda, c)$ is not simple.
\end{lem}
\begin{proof}
By \cite[Theorem 5.1]{LewinSims:MPCPS10} there exists $v \in \Lambda^0$ and an
$\FE(\Lambda)$-compatible filter $S$ such that $v\Lambda s(\lambda) = \emptyset$ for all
$\lambda \in S$ (that is, $v\Lambda s(S) = \emptyset$). Consider
\[\textstyle
	\Hh_S := \clsp\{h_\mu : s(\mu) \in s(S)\}
		\subseteq \ell^2(\Lambda).
\]
Let $\{T_\lambda : \lambda \in \Lambda\}$ be the twisted Toeplitz-Cuntz-Krieger family on
$\ell^2(\Lambda)$ described just before Proposition \ref{prp:nonzero elements}, and let
$\pi_T$ be the associated representation of $\Tt C^*(\Lambda)$. For $\mu,\nu \in \Lambda$
and a basis element $h_\eta$ of $\Hh_S$ we have
\[
T_\mu T^*_\nu h_\eta =
	\begin{cases}
		c(\mu,\eta')\overline{c(\nu,\eta')} h_{\mu\eta'}
			&\text{ if $\eta = \nu\eta'$}\\
		0 &\text{ otherwise.}
	\end{cases}
\]
Since $s(\mu\eta') = s(\eta) \in s(S)$, it follows that $\Hh_S$ is an invariant subspace
of $\ell^2(\Lambda)$ for $\pi_T$. Hence restriction determines a representation
$\pi_T^{\Hh_S}$ of $\Tt C^*(\Lambda, c)$ on $\Hh_S$. We have $T_v h_\eta = 0$ for all
$\eta \in \Lambda s(S)$ because $v\Lambda s(S) = \emptyset$. Thus $s^c_\Tt(v) \in
\ker(\pi_T^{\Hh_S})$.

Let $J = J_{\FE(\Lambda)}$ be the ideal of $\Tt C^*(\Lambda, c)$ generated by
$\{\Delta(s^c_\Tt)^E : E \in \FE(\Lambda)\}$, so that $\Tt C^*(\Lambda, c)/J =
C^*(\Lambda, c)$. Proposition~\ref{prp:->0 on filters} implies that $\lim_{\mu \in S}
\pi_T(a) h_\mu=0$ for all $a \in J$. Since $\|\pi_T(s_\Tt^c(r(S))) h_\mu\| = \|h_{\mu}\|
= 1$ for all $\mu \in S$, we deduce that
\begin{equation} \label{eq:norm 1 quotient}
	\big\|\pi_T(s_\Tt^c(r(S)) - a)|_{\Hh_S}\big\| \ge 1\text{ for all $a \in J$}.
\end{equation}
Let $I := \pi_T^{\Hh_S}(J)$. Then $I$ is an ideal of $\pi_T^{\Hh_S}(\Tt C^*(\Lambda,
c))$. Let $q_I$ be the quotient map from $\pi_T^{\Hh_S}(\Tt C^*(\Lambda, c))$ to
$\pi_T^{\Hh_S}(\Tt C^*(\Lambda, c))/I$. Equation~\eqref{eq:norm 1 quotient} implies that
$\|q_I \circ \pi_T^{\Hh_S}(s_\Tt^c(r(S)))\| = 1$, and so $q_I \circ \pi_T^{\Hh_S}$
induces a nonzero homomorphism $\rho$ of $C^*(\Lambda, c)$. The preceding paragraph shows
that $q_I \circ \pi_T^{\Hh_S}(s_\Tt^c(v)) = 0$. Hence $\rho$ is neither faithful nor
trivial, and it follows that $C^*(\Lambda, c)$ is not simple.
\end{proof}

\begin{prop}\label{prp:expectation}
Let $\Lambda$ be a finitely aligned $k$-graph, and let $c \in \Zcat2(\Lambda, \TT)$.
There is a faithful conditional expectation
\[
\Theta : C^*(\Lambda, c) \to D^c_\Lambda : = \clsp\{s^c(\lambda) s^c(\lambda)^* : \lambda
\in \Lambda\}
\]
given by $\Theta(s^c(\mu) s^c(\nu)^*) = \delta_{\mu,\nu} s^c(\mu) s^c(\mu)^*$.
\end{prop}
\begin{proof}
Averaging over the gauge action $\gamma$ gives a faithful conditional expectation $\Phi :
C^*(\Lambda, c) \to C^*(\Lambda, c)^\gamma = \clsp\{s^c(\mu) s^c(\nu)^* : d(\mu) =
d(\nu)\}$ such that $\Phi(s^c(\mu) s^c(\nu)^*) = \delta_{d(\mu), d(\nu)} s^c(\mu)
s^c(\nu)^*$. Lemma~\ref{lem:mx units} shows that $C^*(\Lambda, c)^\gamma$ can be written
as the closure of an increasing union of finite-dimensional subalgebras $M(s^c)_E$ with
nested diagonal subalgebras whose union is dense in $D^c_\Lambda$.
Equation~\eqref{eq:tmutnu=thetasum} shows that if $\mu \not= \nu$ then
$s^c(\mu)s^c(\nu)^*$ belongs to the span of the off-diagonal matrix units in $M(s^c)_E$
for sufficiently large $E$. So the canonical faithful expectation $\Psi$ from the AF
algebra $C^*(\Lambda, c)^\gamma$ onto its diagonal subalgebra satisfies
$\Theta(s^c(\mu)s^c(\nu)^*) = \delta_{\mu,\nu} s^c(\mu)s^c(\mu)^*$. Now $\Theta = \Psi
\circ \Phi$ is the desired expectation.
\end{proof}

We now restrict our attention to row-finite $k$-graphs with no sources. Recall that in
this setting the Cuntz-Krieger relations for $C^*(\Lambda, c)$ imply that
\begin{equation}\label{eq:rfns CK}
\sum_{\lambda \in v\Lambda^n} s^c(\lambda) s^c(\lambda)^*
    = s^c(v)\text{ for every $v \in \Lambda^0$ and $n \in \NN^k$.}
\end{equation}
It follows from this and~(TCK2) that
\begin{equation}\label{eq:rfnsCK consequence}
s^c(\lambda) s^c(\lambda)^*
	= \sum_{\mu \in \lambda\Lambda^n} s^c(\mu) s^c(\mu)^*\quad\text{ for all $\lambda \in
\Lambda$ and $n \in \NN^k$.}
\end{equation}

\begin{lem}\label{lem:MCE and sim}
Let $\Lambda$ be a row-finite $k$-graph with no sources. Suppose that $\mu \sim \nu$ in
$\Lambda$. Then $\MCE(\mu,\nu) = \mu\Lambda \cap \Lambda^{d(\mu) \vee d(\nu)} =
\nu\Lambda \cap \Lambda^{d(\mu) \vee d(\nu)}$.
\end{lem}
\begin{proof}
For any $\mu\alpha \in \mu\Lambda \cap \Lambda^{d(\mu) \vee d(\nu)}$ and any ultrafilter
$S$ with $s(\alpha) \in S$, we have $\ell_{\mu\alpha}(S) = \ell_\mu(\ell_\alpha(S)) =
\ell_\nu(\ell_\alpha(S)) = \ell_{\nu\alpha}(S)$, and so $\mu\alpha \in \nu\Lambda \cap
\Lambda^{d(\mu) \vee d(\nu)}$. So $\mu\Lambda \cap \Lambda^{d(\mu) \vee d(\nu)} \subseteq
\nu\Lambda \cap \Lambda^{d(\mu) \vee d(\nu)}$, and symmetry gives the reverse inclusion.
Since $\MCE(\mu,\nu) = \mu\Lambda \cap \nu\Lambda \cap \Lambda^{d(\mu) \vee d(\nu)}$, the
result follows.
\end{proof}

Recall from \cite[Theorem~4.2]{CarlsenKangEtAl:xx13} that if $\Lambda$ is cofinal, then
$\Per(\Lambda) = \{d(\mu) - d(\nu) : \mu\sim\nu\}$. Since $\Per(\Lambda)$ it is a
subgroup of $\ZZ^k$, it is isomorphic to $\ZZ^l$ for some $l \le k$. The collection
\begin{align*}
H_{\Per} = \{ v \in \Lambda^0 : \text{ whenever } m,n \in \NN^k&, m-n \in \Per(\Lambda) \text{ and } \lambda \in v\Lambda^m \text{ there } \\
& \text{exists } \mu \in v\Lambda^n \text{ such that } \lambda \sim \mu\}
\end{align*}
is a nonempty hereditary subset of $\Lambda^0$. For $m,n \in \NN^k$ such that $m - n \in
\Per(\Lambda)$, there is a bijection $\theta_{m,n} : \Hper \Lambda^m \to \Hper\Lambda^n$
such that $\mu \sim \theta_{m,n}(\mu)$ for all $\mu$; this $\theta_{m,n}$ preserves the
range and source maps.

Suppose that $\mu \sim \nu \in \Lambda$, and write $m = d(\mu)$ and $n = d(\nu)$. Let $t$
be a Cuntz-Krieger $(\Lambda, c)$-family. Then~\eqref{eq:rfnsCK consequence} implies that
$t_\mu t^*_\mu = \sum_{\eta \in \mu\Lambda^{(m \vee n) - m}} t_\eta t^*_\eta$, and then
Lemma~\ref{lem:MCE and sim} implies that $t_\mu t^*_\mu = t_\mu t^*_\mu t_\nu t^*_\nu$;
that is, $t_\mu t^*_\mu \le t_\nu t^*_\nu$. Symmetry gives the reverse inequality, and
hence
\begin{equation}\label{eq:sim->range projections}
t_\mu t^*_\mu = t_\nu t^*_\nu\qquad\text{ whenever $\mu \sim \nu$.}
\end{equation}

For the following proposition, we remind the reader that $A(c)$ denotes the
noncommutative torus, described at the beginning of this section.

\begin{prop}\label{prp:rotation iso}
Let $\Lambda$ be a row-finite $k$-graph with no sources. Let $c \in Z^2(\ZZ^k, \TT)$, and
suppose that $(cc^*)|_{\Per(\Lambda)}$ is nondegenerate. Suppose that $t$ is a
Cuntz-Krieger $(\Lambda, c \circ d)$-family such that each $t_v \not= 0$. Suppose that
$F$ is a finite subset of $\Hper\Lambda^n$, and let $\widetilde{F} := \{\nu \in \Lambda :
\nu \sim \mu\text{ for some } \mu \in F\}$. There is an injective homomorphism
$\bigoplus_{\lambda \in F} \rho^\lambda_t : \bigoplus_{\lambda \in F}
A(c|_{\Per(\Lambda)}) \to C^*(t)$ such that $\rho^\lambda_t(U_{d(\lambda)- d(\nu)}) =
t_\lambda t^*_\nu$ whenever $\lambda \in F$ and $\nu \sim \lambda$.
\end{prop}
\begin{proof}
As discussed at the beginning of the section, we may assume that $c$ is a bicharacter.

If $\mu,\nu \in \widetilde{F}$ and $\mu\not\sim\nu$, then there exist distinct
$\lambda,\lambda' \in F$ such that $\mu \sim \lambda$ and $\nu\sim\lambda'$.
Equation~\eqref{eq:sim->range projections} gives $t_\mu t^*_\mu = t_\lambda t^*_\lambda
\perp t_{\lambda'} t^*_{\lambda'} = t_\nu t^*_\nu$. Hence
\[
C^*(\{t_\mu t^*_\nu : \mu,\nu \in \widetilde{F}\text{ and } \mu \sim \nu\})
	= \bigoplus_{\lambda \in F} C^*(\{t_\mu t^*_\nu : \mu \sim \lambda \sim \nu\}).
\]
Since each $t_v$ is nonzero, so is each $t_\mu t^*_\nu$, and so each summand $C^*(\{t_\mu
t^*_\nu : \mu \sim \lambda \sim \nu\})$ is nontrivial.

Fix $\lambda \in F$. It now suffices to show that there is an injective homomorphism
$\rho_s : A(c|_{\Per(\Lambda)}) \to t_\lambda t^*_\lambda C^*(t) t_\lambda t^*_\lambda$
such that $\rho_s(U_{d(\mu) - d(\nu)}) = t_\mu t^*_\nu$ whenever $\mu \sim \lambda \sim
\nu$. To see this, for each $m \in \Per(\Lambda)$, express $m = p(m) - q(m)$ where $p(m),
q(m) \in \NN^k$ and $p(m) = n$ whenever $m \le n$. Since $r(\lambda) \in \Hper$, we can
define
\begin{equation}\label{eq:Vm}
V_m := \sum_{\mu \in r(\lambda)\Lambda^{p(m)}} t_\lambda t^*_\lambda t_\mu
t^*_{\theta_{p(m),q(m)}(\mu)}.
\end{equation}
We have
\[
V_m V^*_m = \sum_{\mu,\nu \in r(\lambda)\Lambda^{p(m)}} t_\lambda t^*_\lambda t_\mu
t^*_{\theta_{p(m),q(m)}(\mu)}
	t_{\theta_{p(m),q(m)}(\nu)} t^*_\nu t_\lambda t^*_\lambda.
\]
Each $d(\theta_{p(m),q(m)}(\mu)) = q(m) = d(\theta_{p(m),q(m)}(\nu))$ and
$\theta_{p(m),q(m)}$ is a bijection, so the Cuntz-Krieger relation~\eqref{eq:rfns CK}
gives
\[
V_m V_m^* = \sum_{\mu \in r(\lambda)\Lambda^p} t_\lambda t^*_\lambda t_\mu t^*_\mu
t_\lambda t^*_\lambda,
\]
and then~\eqref{eq:rfns CK} again reduces this to $t_\lambda t^*_\lambda$.

To see that $V_m^* V_m$ is also equal to $t_\lambda t^*_\lambda$, observe that
\begin{align*}
V^*_m V_m
	&= \sum_{\mu,\nu \in r(\lambda)\Lambda^{p(m)}} t_{\theta_{p(m),q(m)}(\mu)} t^*_\mu t_\lambda t^*_\lambda t_\nu t^*_{\theta_{p(m),q(m)}(\nu)}\\
	&= \sum_{\substack{\mu,\nu \in r(\lambda)\Lambda^{p(m)}\\ \eta \in
\lambda\Lambda^{p(m)}}}
		 t_{\theta_{p(m),q(m)}(\mu)} t^*_\mu t_\eta t^*_\eta t_\nu
t^*_{\theta_{p(m),q(m)}(\nu)}.
\end{align*}
For each $\eta \in \lambda\Lambda^{p(m)}$, factorise $\eta = \alpha_\eta\beta_\eta$ with
$d(\alpha_\eta) = p(m)$ and $d(\beta_\eta) = n$. For $\mu \in r(\lambda)\Lambda^{p(m)}$
and $\eta \in \lambda \Lambda^{p(m)}$ we have $d(\mu) = p(m) = d(\alpha_\eta)$, so the
first displayed equation in Lemma~\ref{lem:TCK consequences} implies that $t^*_\mu
t_{\alpha_\eta} = \delta_{\mu, \alpha_\eta} t_{s(\mu)}$. Thus
\[
V^*_m V_m
	= \sum_{\eta \in \lambda\Lambda^{p(m)}} t_{\theta_{p(m),q(m)}(\alpha_\eta)}
t_{\beta_\eta}
		t^*_{\beta_\eta} t^*_{\theta_{p(m),q(m)}(\alpha_\eta)}.
\]
If $\mu \sim \nu$ then the factorisation property gives $\mu\tau \sim \nu\tau$ for all
$\tau$, and so $\theta_{m + p, n + p}(\mu\tau) = \theta_{m,n}(\mu)\tau$ whenever $\mu \in
\Lambda^m$ and $\tau \in s(\mu)\Lambda^p$. This and~\eqref{eq:sim->range projections}
give
\begin{equation}\label{unitary_SP_formula}
\begin{split}
V^*_m V_m
	&= \sum_{\eta \in \lambda\Lambda^{p(m)}} t_{\theta_{p(m)+n,q(m)+n}(\eta)} t^*_{\theta_{p(m)+n, q(m)+n}(\eta)} \\
	&= \sum_{\eta \in \lambda\Lambda^{p(m)}} t_\eta t^*_\eta
	= t_\lambda t^*_\lambda.
\end{split}
\end{equation}
Hence the $V_m$ are unitaries in $t_\lambda t^*_\lambda C^*(t) t_\lambda t^*_\lambda$.

Combining \eqref{eq:Vm} and \eqref{unitary_SP_formula} we now have
\begin{align*}
\sum_{\mu \in r(\lambda)\Lambda^{p(m)}} t_\lambda t^*_\lambda t_\mu
t^*_{\theta_{p(m),q(m)}(\mu)}
	&= V_m = V_m t_\lambda t^*_\lambda\\
    &= \sum_{\mu \in r(\lambda)\Lambda^{p(m)}}
t_\lambda t^*_\lambda t_\mu t^*_{\theta_{p(m),q(m)}(\mu)} t_\lambda t^*_\lambda,
\end{align*}
and then symmetry gives
\[
V_m =  \sum_{\mu \in r(\lambda)\Lambda^{p(m)}} t_\mu t^*_{\theta_{p(m),q(m)}(\mu)}
t_\lambda t^*_\lambda.
\]

Fix $m, m'$ in $\Per(\Lambda)$, and write $p := p(m)$, $p' := p(m')$, $q := q(m)$ and $q'
:= q(m')$. For the next few calculations, put
\[
C_{m,m'} := c(p,p')\overline{c(q,p')}c(p', q)\overline{c(q',q)}.
\]
Then relation~\eqref{eq:rfns CK} gives
\begin{align*}
V_m V_{m'}
	&= \sum_{\mu \in r(\lambda)\Lambda^p} t_\lambda t^*_\lambda t_\mu
t^*_{\theta_{p,q}(\mu)}
		\sum_{\eta \in r(\lambda)\Lambda^{p'}} t_\eta t^*_{\theta_{p',q'}(\eta)} t_\lambda t^*_\lambda \\
	&=  \sum_{\substack{\mu \in r(\lambda)\Lambda^p,\;
						\mu' \in s(\mu)\Lambda^{p'}\\
						\eta \in r(\lambda)\Lambda^{p'},\;
						\eta' \in s(\eta)\Lambda^q}}
			C_{m,m'}
			t_\lambda t^*_\lambda t_{\mu\mu'} t^*_{\theta_{p,q}(\mu)\mu'}
			t_{\eta\eta'} t^*_{\theta_{p',q'}(\eta)\eta'} t_\lambda t^*_\lambda\\
	&= C_{m,m'}
		\sum_{\substack{\mu \in r(\lambda)\Lambda^p,\;
							\mu' \in s(\mu)\Lambda^{p'}\\
							\eta \in r(\lambda)\Lambda^{p'},\;
							\eta' \in s(\eta)\Lambda^q}}
				\delta_{\theta_{p,q}(\mu)\mu', \eta\eta'}
				t_\lambda t^*_\lambda t_{\mu\mu'} t^*_{\theta_{p',q'}(\eta)\eta'}
t_\lambda t^*_\lambda.
\end{align*}
If $\theta_{p,q}(\mu)\mu' = \eta\eta'$, then $\mu\mu' \sim \theta_{p,q}(\mu)\mu' =
\eta\eta' \sim \theta_{p',q'}(\eta)\eta'$; and conversely $\mu\mu' \sim
\theta_{p',q'}(\eta)\eta'$ implies $\theta_{p,q}(\mu)\mu' = \eta\eta'$ by a symmetric
argument. So the final line of the preceding displayed equation becomes
\begin{equation}\label{eq:convert to a+b}
C_{m,m'}
		\sum_{\substack{\mu \in r(\lambda)\Lambda^p,\;
							\mu' \in s(\mu)\Lambda^{p'}\\
							\zeta \in r(\lambda)\Lambda^{q'},\;
							\zeta' \in s(\eta)\Lambda^q\\
							\mu\mu' \sim \zeta\zeta'}}
				t_\lambda t^*_\lambda t_{\mu\mu'} t^*_{\zeta\zeta'} t_\lambda
t^*_\lambda.
\end{equation}
Let $a := p(m + m')$ and $b := q(m+m')$. Then there are $h,l \in \NN^k$ such that $p + p'
+ h = a+l$ and $q + q' + h = b + l$. For $\alpha \in \Lambda^{p+p'}$ and $\beta \in
\Lambda^{q + q'}$, we have $\alpha \sim \beta$ if and only if there is an injection $\tau
\mapsto (\eta(\tau), \rho(\tau), \zeta(\tau))$ from $s(\alpha)\Lambda^h$ to
$r(\alpha)\Lambda^a \times \Lambda^l s(\alpha) \times r(\alpha)\Lambda^b$ such that
$\alpha\tau = \eta(\tau)\rho(\tau)$ and $\beta\tau = \zeta(\tau) \rho(\tau)$ and
$\eta(\tau) \sim \zeta(\tau)$ for each $\tau$. Using this, and applying
relation~\eqref{eq:rfns CK}, we obtain
\begin{align*}
\sum_{\substack{\mu \in r(\lambda)\Lambda^p,\;
							\mu' \in s(\mu)\Lambda^{p'}\\
							\zeta \in r(\lambda)\Lambda^{q'},\;
							\zeta' \in s(\eta)\Lambda^q\\
							\mu\mu' \sim \zeta\zeta',}}&
				t_{\mu\mu'} t^*_{\zeta\zeta'}\\
	&= \sum_{\substack{\mu \in r(\lambda)\Lambda^p,\;
								\mu' \in s(\mu)\Lambda^{p'}\\
								\zeta \in r(\lambda)\Lambda^{q'},\;
								\zeta' \in s(\eta)\Lambda^q\\
								\mu\mu' \sim \zeta\zeta',\;
								\tau \in s(\mu')\Lambda^h}}
					 c(p+p',h)\overline{c(q+q',h)} t_{\mu\mu'\tau} t^*_{\zeta\zeta'\tau}\\
	&= \sum_{\substack{\mu \in r(\lambda)\Lambda^p,\;
								\mu' \in s(\mu)\Lambda^{p'}\\
								\zeta \in r(\lambda)\Lambda^{q'},\;
								\zeta' \in s(\eta)\Lambda^q\\
								\mu\mu' \sim \zeta\zeta',\;
								\tau \in s(\mu')\Lambda^h}}
					 c(m+m',h) t_{\mu\mu'\tau} t^*_{\zeta\zeta'\tau}.		
\end{align*}
Hence the expression~\eqref{eq:convert to a+b} for $V_m V_m^*$ gives
\begin{align*}
V_m V_{m'}
	&= C_{m,m'} c(m+m',h)
		\sum_{\eta \in r(\lambda)\Lambda^a\!,\, \rho \in s(\eta)\Lambda^l}
				t_\lambda t^*_\lambda t_{\eta\rho} t^*_{\theta_{a,b}(\eta)\rho} t_\lambda t^*_\lambda \\
	&= C_{m,m'}  c(m+m',h) \overline{c(a,l)}c(b,l) \\
	&\phantom{C_{m,m'}}
			\sum_{\eta \in r(\lambda)\Lambda^a\!,\, \rho \in s(\eta)\Lambda^l}
					t_\lambda t^*_\lambda t_{\alpha'} t_\tau t^*_\tau t^*_{\theta_{a,b}(\alpha')} t_\lambda t^*_\lambda \\
	&= c(p,p')\overline{c(q,p')}c(p',q)\overline{c(q',q)} c(m+m',h-l) V_{m + m'}.
\end{align*}
Rearranging this expression for $V_m V_{m'}$ and the symmetric expression
\[
V_{m'} V_m
	=  c(p',p)\overline{c(q',p)}c(p, q')\overline{c(q,q')} c(m+m',h-l) V_{m + m'},
\]
and cancelling the occurrences of $c(m+m',h-l)$, we obtain
\begin{equation}\label{eq:commutation}
\begin{split}
V_m V_{m'}
	= c(p,p')&\overline{c(q,p')} \overline{c(p,q')}c(q,q')\\
        &\overline{c(p',p)} c(p',q) \overline{c(q',q)} c(q',p) V_{m'}V_m.
\end{split}
\end{equation}
Using repeatedly that $c$ is a bicharacter, we calculate:
\begin{align*}
c(p,p')\overline{c(q,p')}&
		\overline{c(p,q')}c(q,q') \overline{c(p',p)} c(p',q) \overline{c(q',q)} c(q',p)\\
	&= c(p-q, p') c(-p+q, q') c(p', -p+q) c(q', p-q) \\
	&= c(p-q, p') c(p-q, -q') \overline{c(p', p-q)c(-q',p-q)}\\
	&= c(m, m')\overline{c(m',m)}.
\end{align*}
Thus~\eqref{eq:commutation} becomes
\[
V_m V_{m'} = c(m, m')\overline{c(m',m)} V_{m'}V_m.
\]
Choose generators $g_1, \dots, g_l$ for $\Per(\Lambda)$. Then each $V_m$ is a scalar
multiple of a product of the $V_{g_i}$, and so $C^*(\{V_m : m \in \Per(\Lambda)\})$ is
generated by the $V_{g_i}$, which satisfy $V_{g_i} V_{g_j} = (cc^*)(g_i, g_j) V_{g_j}
V_{g_i}$. Since $c_{\Per} := c|_{\Per(\Lambda)}$ is nondegenerate,
\cite[Theorem~3.7]{Slawny:CMP72} implies that there is an isomorphism $\rho^\lambda_t :
A(c_{\Per}) \to C^*(\{V_m : m \in \Per(\Lambda)\})$ carrying each $U_m$ to $V_m$.

Recall that we chose $p(m) = n$ whenever $m \le n$. So if $\lambda \sim \nu$ then
\[
V_{n - d(\nu)}
	= t_\lambda t^*_\lambda \sum_{\mu \in r(\lambda)\Lambda^n} t_\mu t^*_{\theta_{n,
d(\nu)}(\mu)}
	= t_\lambda t^*_{\theta_{n, d(\nu)}(\lambda)}
	= t_\lambda t^*_\nu.
\]
So $\rho^\lambda_t$ carries $U_{d(\lambda)- d(\nu)}$ to $t_\lambda t^*_\nu$ whenever $\nu
\sim \lambda$. Hence $\bigoplus_{\lambda \in F} \rho^\lambda_t : \bigoplus_{\lambda \in
F} A(c_{\Per}) \to C^*(t)$ is the desired homomorphism.
\end{proof}

\begin{lem}\label{eq:sim vs MCE}
Let $\Lambda$ be a row-finite $k$-graph with no sources. Suppose that $\mu,\nu \in
\Lambda$ satisfy $s(\mu) = s(\nu)$. Then $\mu \sim \nu$ if and only if
$\MCE(\mu\tau,\nu\tau) \not= \emptyset$ for all $\tau \in s(\mu)\Lambda$.
\end{lem}
\begin{proof}
First suppose that $\mu \sim \nu$, and fix $\tau \in s(\mu)\Lambda$. There is an
ultrafilter $S$ with $\tau \in S$, and we have $\ell_\mu(S) = \ell_\nu(S)$. Since $S$ is
an ultrafilter, so is $\ell_\mu(S)$, and so $\ell_\mu(S) \cap \Lambda^{d(\mu\tau) \vee
d(\nu\tau)} \not= \emptyset$, and then the unique element of $\ell_\mu(S) \cap
\Lambda^{d(\mu\tau) \vee d(\nu\tau)}$ belongs to $\MCE(\mu\tau, \nu\tau)$.

Now suppose that $\MCE(\mu\tau, \nu\tau) \not= \emptyset$ for all $\tau$. Using the
correspondence between ultrafilters and infinite paths of Remark~\ref{rmk:filters vs
paths} together with the topology on the space of infinite paths described in
\cite[Definition~2.4 and Lemma~2.6]{KumjianPask:NYJM00}, we see that the sets $Z(\lambda)
= \{S : S\text{ is an ultrafilter and }\lambda \in S\}$ are a basis of compact open sets
for a Hausdorff topology on the set of all ultrafilters. Fix an ultrafilter $S$ with
$s(\mu) \in S$. For each $\tau \in S$, we have $\MCE(\mu\tau, \nu\tau) \not= \emptyset$,
and hence $Z(\mu\tau) \cap Z(\nu\tau) \not= \emptyset$. Since the  $Z(\mu\tau) \cap
Z(\nu\tau)$ are compact and decreasing with respect to the partial ordering $\leq$ on
$S$, we deduce that $\bigcap_{\tau \in S} Z(\mu\tau) \cap Z(\nu\tau) \not= \emptyset$;
say $T \in \bigcap_{\tau \in S} Z(\mu\tau) \cap Z(\nu\tau)$. For $\tau \le \tau' \in S$,
we have $S' \cap \Lambda^{d(\mu\tau)} = \{\mu\tau\}$ for all $S' \in Z(\mu\tau')$. Hence
$T \cap \Lambda^{d(\mu\tau)} \subseteq \{\mu\tau\}$ for all $\tau \in S$. Since $T \cap
\Lambda^n \not= \emptyset$ for all $n$, we deduce that $T \cap \Lambda^{d(\mu\tau)} =
\{\mu\tau\}$ for all $\tau \in S$. So $\ell_\mu(S) \subseteq T$, and since $\ell_\mu(S)$
is an ultrafilter, the two are equal. Symmetry gives $\ell_\nu(S) = T$ as well, so
$\ell_\mu(S) = \ell_\nu(S)$.
\end{proof}

\begin{prop}\label{prp:norm-decreasing}
Let $\Lambda$ be a row-finite $k$-graph with no sources. Let $c \in Z^2(\ZZ^k, \TT)$ and
suppose that $(cc^*)|_{\Per(\Lambda)}$ is nondegenerate. Suppose that $t$ is a
Cuntz-Krieger $(\Lambda, c \circ d)$-family with each $t_v$ nonzero. Suppose that $a :
\Hper\Lambda \times \Hper\Lambda \to \CC$ is finitely supported. Then
\[
\Big\|\sum_{\mu,\nu \in \Lambda} a_{\mu,\nu} t_\mu t^*_\nu\Big\|
    \ge \Big\|\sum_{\mu \sim \nu} a_{\mu,\nu} t_\mu t^*_\nu\Big\|.
\]
\end{prop}
\begin{proof}
Again as discussed at the beginning of the section, we may assume that $c$ is a
bicharacter.

Let $N = \bigvee_{a_{\mu,\nu} \not= 0} d(\mu)$. Using the Cuntz-Krieger relation, we can
rewrite
\[
\sum_{\mu,\nu} a_{\mu,\nu} t_\mu t^*_\nu
	= \sum_{\mu,\nu}\sum_{\mu\mu' \in \Lambda^N} a_{\mu,\nu}
        c(\mu,\mu')\overline{c(\nu,\mu')} t_{\mu\mu'} t^*_{\nu\mu'};
\]
so we may assume without loss of generality that $a_{\mu,\nu} \not= 0$ implies $\mu \in
\Lambda^N$. Since the $t_\mu t^*_\mu$ for $\mu \in \Lambda^N$ are mutually orthogonal,
and since $\mu \sim \nu$ implies $t_\mu t^*_\mu = t_\nu t^*_\nu$, we have $\clsp\{t_\mu
t^*_\nu : \mu \in \Lambda^N, \nu\sim \mu\} = \bigoplus_{\mu \in \Lambda^N} \clsp\{t_\mu
t^*_\nu : \mu \sim \nu\}$, and so $\big\|\sum_{\mu \sim \nu} a_{\mu,\nu} t_\mu
t^*_\nu\big\| = \max_{\mu \in \Lambda^N} \big\|\sum_{\nu \sim \mu} a_{\mu,\nu} t_\mu
t^*_\nu\big\|$.

Choose $\mu_{\max}$ such that
\[
\Big\|\sum_{\nu \sim \mu_{\max}} a_{\mu_{\max},\nu} t_{\mu_{\max}} t^*_\nu\Big\|
    \ge \Big\|\sum_{\nu' \sim \mu'} a_{\mu',\nu'} t_{\mu'} t^*_{\nu'}\Big\|
\]
for all $\mu' \in \Lambda^N$. Let $v_{\max} = s(\mu_{\max})$. We claim that if $\mu,\nu
\in \Lambda$ satisfy $s(\mu) = s(\nu) = v_{\max}$ but $\mu \not\sim \nu$, then there
exists $\tau \in v_{\max}\Lambda$ such that $\MCE(\mu\tau, \nu\tau) = \emptyset$. Indeed,
Lemma~\ref{eq:sim vs MCE} implies that either there exists $\tau \in
v_{\max}\Lambda^{d(\mu) \vee d(\nu) - d(\mu)}$ such that $\mu\tau \not\in \MCE(\mu,\nu)$
or there exists $\tau \in v_{\max}\Lambda^{d(\mu)\vee d(\nu) - d(\nu)}$ such that
$\nu\tau \not\in \MCE(\mu,\nu)$. We consider the first case; the second is symmetric. We
have $\mu\tau \in \Lambda^{d(\mu) \vee d(\nu)}$, and so $\mu\tau = \alpha\alpha'$ for
some $\alpha \in \Lambda^{d(\nu)} \setminus\{\nu\}$. In particular, factoring $\mu\tau =
\eta\zeta$ with $d(\eta) = d(\nu)$, we have $\eta \not= \nu$, and so
$\MCE(\mu\tau,\nu\tau) = \emptyset$. An induction using that $\Lambda$ has no sources now
shows that there exists $\tau \in v_{\max}\Lambda$ such that $d(\tau) > d(\mu) \vee
d(\nu)$ whenever $a_{\mu,\nu} \not= 0$ and $\MCE(\mu\tau,\nu\tau) = \emptyset$ whenever
$s(\mu) = s(\nu) = v_{\max}$ and $\mu \not\sim \nu$.

Let $l := d(\tau)$. For $\mu,\nu \in F$ with $a_{\mu,\nu} \not= 0$, we have
\begin{align*}
\sum&_{\lambda \in Fv_{\max} \cap \Lambda^N} t_{\lambda\tau} t_{\lambda\tau}^*
		t_\mu t^*_\nu t_{\lambda\tau} t_{\lambda\tau}^* \\
	&= \begin{cases}
		c(N, l)\overline{c(d(\nu), l)} t_{\mu\tau}t^*_{\nu\tau} t_{\mu\tau} t_{\mu\tau}^*
			&\text{ if $s(\mu) = v_{\max}$ and $\mu \sim \nu$}\\
		0 &\text{ otherwise.}
	\end{cases}
\end{align*}
If $\mu \sim \nu$ then $\mu\tau \sim \nu\tau$ giving $t_{\mu\tau} t_{\mu\tau}^* =
t_{\nu\tau} t_{\nu\tau}^*$, and hence $t^*_{\nu\tau} t_{\mu\tau} t_{\mu\tau}^* =
t^*_{\nu\tau}$. That is,
\begin{align*}
\sum&_{\lambda \in Fv_{\max} \cap \Lambda^N} t_{\lambda\tau} t_{\lambda\tau}^*
		t_\mu t^*_\nu t_{\lambda\tau} t_{\lambda\tau}^*\\
	&= \begin{cases}
		c(N, l)\overline{c(d(\nu), l)} t_{\mu\tau}t^*_{\nu\tau}
			&\text{ if $s(\mu) = v_{\max}$ and $\mu \sim \nu$}\\
		0 &\text{ otherwise.}
	\end{cases}
\end{align*}
Since $c$ is a bicharacter, the map $m \mapsto c(m,l)$ is a homomorphism of
$\Per(\Lambda)$ into $\TT$. The universal property of $A(c|_{\Per(\Lambda)})$ ensures
that it admits an automorphism $\alpha_l$ such that $\alpha_l(U_m) = c(m,l)U_m$ for all
$m \in \Per(\Lambda)$. An application of Proposition~\ref{prp:rotation iso} with $F =
\{\mu_{\max}\tau\}$ gives an injective homomorphism $\rho_1 : A(c|_{\Per(\Lambda)}) \to
C^*(t)$ which carries $U_{N-d(\nu)}$ to $t_{\mu_{\max}\tau} t^*_{\nu\tau}$ whenever $\nu
\sim \mu_{\max}$; so $\rho_1 \circ \alpha_l$ carries $U_{N-d(\nu)}$ to $\sum_{\lambda \in
Fv_{\max} \cap \Lambda^N} t_{\lambda\tau} t_{\lambda\tau}^*$.
Proposition~\ref{prp:rotation iso} applied to $F = \{\mu_{\max}\}$ gives an injective
homomorphism $\rho_2 : A(c|_{\Per(\Lambda)}) \to C^*(t)$ which carries $U_{N- d(\nu)}$ to
$t_{\mu_{\max}} t^*_{\nu}$ whenever $\nu \sim \mu_{\max}$. Since these homomorphisms are
injective, they are isometric, and so the map
\[
\rho_2 \circ \alpha_l^{-1} \circ \rho_1^{-1} : \clsp\{t_{\mu_{\max}\tau} t^*_{\nu\tau} :
\nu \sim \mu_{\max}\} \to \clsp\{t_{\mu_{\max}} t^*_\nu : \nu \sim \mu_{\max}\}
\]
is isometric and carries spanning elements to spanning elements. Since the
$t_{\lambda\tau} t^*_{\lambda\tau}$ are mutually orthogonal projections, the map $a
\mapsto \sum_{\lambda} t_{\lambda\tau} t^*_{\lambda\tau} a t_{\lambda\tau}
t^*_{\lambda\tau}$ is norm-decreasing, and so we have
\begin{align*}
\Big\|\sum_{\mu \sim \nu}& a_{\mu,\nu} t_\mu t^*_\nu\Big\| \\
	&= \Big\|\sum_{\nu \sim \mu_{\max}} a_{\mu_{\max},\nu} t_{\mu_{\max}} t^*_\nu\Big\| \\
	&= \Big\|\rho_2 \circ \alpha_l^{-1} \circ \rho_1^{-1}\Big(\sum_{\lambda \in Fv_{\max} \cap \Lambda^n} t_{\lambda\tau} t_{\lambda\tau}^* \Big(\sum_{\nu \sim \mu_{\max}} a_{\mu_{\max},\nu} t_{\mu_{\max}} t^*_\nu\Big) t_{\lambda\tau} t_{\lambda\tau}^*\Big)\Big\|\\
	&= \Big\|\sum_{\lambda \in Fv_{\max} \cap \Lambda^n} t_{\lambda\tau}
t_{\lambda\tau}^*
			\sum_{\mu,\nu \in \Lambda} a_{\mu,\nu} t_\mu t^*_\nu t_{\lambda\tau} t_{\lambda\tau}^*\Big\| \\
	&\le \Big\|\sum_{\mu,\nu \in \Lambda} a_{\mu,\nu} t_\mu t^*_\nu\|,
\end{align*}
as required.
\end{proof}

\begin{proof}[Proof of Theorem~\ref{thm:simple}]
Since $\Lambda$ is cofinal, the saturation of $\Hper$ is all of $\Lambda^0$. Hence
Lemma~\ref{lem:ideal sets} implies that $P_{\Hper} C^*(\Lambda, c \circ d) P_{\Hper}$ is
a full corner of $C^*(\Lambda, c \circ d)$. So it suffices to show that $P_{\Hper}
C^*(\Lambda, c \circ d) P_{\Hper}$ is simple. Let $\Gamma = \Hper \Lambda$. The
gauge-invariant uniqueness theorem gives a canonical isomorphism $C^*(\Gamma, c \circ d)
\cong P_{\Hper} C^*(\Lambda, c \circ d) P_{\Hper}$. So it suffices to show that
$C^*(\Gamma, c \circ d)$ is simple.

Suppose $\pi$ is a homomorphism $C^*(\Gamma, c\circ d) \to B$, and let $s := \{s_\lambda
: \lambda \in \Gamma\}$ be the universal Cuntz-Krieger $(\Gamma, c \circ d)$-family.
First suppose that $\pi(s_v) = 0$ for some $v$. Then $\ker(\pi)$ contains the
gauge-invariant ideal generated by $s_v$, which in turn, by Theorem~\ref{thm:g-i ideals},
contains $s_w$ for every $w$ in the saturated hereditary set generated by $v$. Since
$\Gamma$ is cofinal it follows that $\pi(s_w) = 0$ for all $w$, and then $\pi(s_\mu
s^*_\nu) = \pi(s_{r(\mu)} s_\mu s^*_\nu) = 0$ for all $\mu,\nu$, and hence $\pi = 0$. Now
suppose that $\pi(s_v) \not = 0$ for all $v$; we must show that $\pi$ is injective. Let
$t_\mu := \pi(s_\mu)$ for all $\mu$. Since $\Hper$ is all of $\Gamma^0$,
Proposition~\ref{prp:norm-decreasing} implies that $\big\|\sum a_{\mu,\nu} t_\mu
t^*_\nu\big\| \ge \big\|\sum_{\mu\sim\nu} a_{\mu,\nu} t_\mu t^*_\nu\big\|$ for all
finitely supported $a$, and so there is a norm-decreasing linear map $\Psi : C^*(t) \to
\clsp\{t_\mu t^*_\nu : \mu \sim \nu\}$ such that $\Psi(t_\mu t^*_\nu) = t_\mu t^*_\nu$ if
$\mu \sim \nu$ and is zero otherwise. The same argument gives a norm-decreasing linear
map $\Phi$ from $C^*(\Gamma, c\circ d)$ to $\clsp\{s_\mu s^*_\nu : \mu \sim \nu\}$. This
is faithful on positive elements because the faithful conditional expectation $\Theta$ of
Proposition~\ref{prp:expectation} satisfies $\Theta = \Theta \circ \Psi$. Clearly $\pi
\circ \Phi = \Psi \circ \pi$. Fix a finite linear combination $\sum_{\mu \sim \nu}
a_{\mu,\nu} s_\mu s^*_\nu$. Using the Cuntz-Krieger relation as in the proof of
Proposition~\ref{prp:norm-decreasing}, we may assume that there exists $n \in \NN^k$ such
that $a_{\mu,\nu} \not= 0$ implies that $\mu \in \Lambda^n$. Applying
Proposition~\ref{prp:rotation iso} to the Cuntz-Krieger $(\Gamma, c\circ d)$-families $s$
and $t$ and with $F = \{\mu : a_{\mu,\nu} \not= 0\}$ and composing the resulting
homomorphisms, we obtain an isometric map from $C^*(\{s_\mu s^*_\nu : \lambda \in F,
\mu\sim \lambda \sim \nu\})$ to $C^*(\{t_\mu t^*_\nu : \lambda \in F, \mu \sim \lambda
\sim \mu\})$ which carries each $s_\mu s^*_\nu$ to $t_\mu t^*_\nu$. Hence
$\big\|\sum_{\mu\sim\nu} a_{\mu,\nu} s_\mu s^*_\nu\big\| = \big\|\sum_{\mu\sim\nu}
a_{\mu,\nu} t_\mu t^*_\nu\big\|$. Since $a$ was arbitrary, we deduce that
$\pi|_{\clsp\{s_\mu s^*_\nu : \mu \sim \nu\}}$ is injective. So Lemma~\ref{lem:usual
argument} shows that $\pi$ is injective.
\end{proof}

\Addresses
\end{document}